\documentclass[11pt]{article}

\usepackage{float,enumitem,amsthm,amsmath,amsfonts,amssymb,graphicx,latexsym,mathrsfs}

\usepackage{graphicx}

\usepackage{fullpage,setspace}

\makeatletter
\newtheorem*{rep@theorem}{\rep@title}
\newcommand{\newreptheorem}[2]{%
\newenvironment{rep#1}[1]{%
 \def\rep@title{#2 \ref{##1}}%
 \begin{rep@theorem}}%
 {\end{rep@theorem}}}
\makeatother

\theoremstyle{plain}
\newtheorem{theorem}{Theorem}[section]
\newtheorem{lemma}[theorem]{Lemma}
\newreptheorem{lemma}{Lemma}
\newtheorem{proposition}[theorem]{Proposition}
\newreptheorem{proposition}{Proposition}
\newtheorem{corollary}[theorem]{Corollary}

\newtheorem{scenario}{Scenario}
\theoremstyle{definition}

\newtheorem{example}{Example}[section]
\theoremstyle{remark}

\newcommand{\expect}[1]{{\mathbb E}\{{ #1 }\}}
\newcommand{\pr}[1]{{\mathbb P}\{{ #1 }\}}
\newcommand{\e}{{\rm e}}
\newcommand{\eqd}{{\,{\buildrel d \over =}\,}}
\newcommand{\stl}{\geq_{{\rm st}}}
\newcommand{\stsl}{>_{{\rm st}}}
\newcommand{\cond}{\xrightarrow{\;d\;}}
\newcommand{\stat}[2]{\pi(#1,#2)}
\newcommand{\TP}{L}
\newcommand{\AP}{\phi}
\newcommand{\DP}{\psi}
\begin{document}

\title{Queues with Random Back-Offs}
\author{N. Bouman\thanks{Eindhoven University of Technology, P.O. Box 513, 5600 MB Eindhoven, The Netherlands.} \and S.C. Borst\footnotemark[1] \thanks{Alcatel-Lucent Bell Labs, P.O. Box 636, Murray Hill, NJ 07974-0636, USA.} \and O.J. Boxma\footnotemark[1] \and J.S.H. van Leeuwaarden\footnotemark[1]}

\date{}

\maketitle

\begin{abstract}
We consider a broad class of queueing models with random
state-dependent vacation periods, which arise in the analysis of
queue-based back-off algorithms in wireless random-access networks.
In contrast to conventional models, the vacation periods may be
initiated after each service completion, and can be randomly terminated
with certain probabilities that depend on the queue length.
We examine the scaled queue length and delay in a heavy-traffic regime,
and demonstrate a sharp trichotomy, depending on how the activation
rate and vacation probability behave as function of the queue length.
In particular, the effect of the vacation periods may either
(i) completely vanish in heavy-traffic conditions, (ii) contribute an additional term to the queue lengths and delays of similar magnitude, or even (iii) give
rise to an order-of-magnitude increase.
The heavy-traffic asymptotics are obtained by combining stochastic
lower and upper bounds with exact results for some specific cases.
The heavy-traffic trichotomy provides valuable insight in the impact of the back-off
algorithms on the delay performance in wireless random-access networks.

\end{abstract}

\section{Introduction}
\label{intro}
We consider a broad class of queueing models with random
state-dependent vacation periods.
In contrast to conventional vacation models (see for instance~\cite{Tak91} for a comprehensive overview), the server may take
a vacation after each service completion and return from a vacation
with certain probabilities that depend on the queue length.
Specifically, when there are $i$~customers left behind in the system
after a service completion, the server either takes a vacation with
probability $\psi(i)$, with $\psi(0) \equiv 1$, or starts the
service of the next customer otherwise.
Likewise, the server returns from a vacation and starts the service of
the next customer at the first event of a non-homogeneous Poisson
process of rate $f(i)$, with $f(0) \equiv 0$, when there are
$i$~customers in the system.

In view of the vacation discipline, we analyze the queue length process
at departure epochs, unlike most papers on vacation models which
consider the queue length process embedded at instants when vacations
begin or end.
A notable exception is \cite{HM88}, which studies an M/G/1 queue
with a similar state-dependent vacation discipline, and establishes
a stochastic decomposition property under certain assumptions. We show that this decomposition property in fact holds in far
greater generality and corresponds to the Fuhrmann-Cooper decomposition.
In addition, we obtain the exact stationary distributions of the queue
length and delay for M/G/1 queues in three scenarios:
(i) the probability $\psi(\cdot)$ decays geometrically as a function of
the queue length, and the vacation is independent of the queue length;
(ii) the probability $\psi(\cdot)$ is inversely proportional to the queue
length, and the vacation is independent of the queue length;
(iii) $\psi(\cdot) \equiv 1$ and the activation rate $f(\cdot)$ is
proportional to the queue length.

We further derive lower and upper bounds for the mean queue length
and mean delay in two cases:
the activation rate $f(\cdot)$ is fixed and the vacation probability
$\psi(\cdot)$ is a convex decreasing function;
the activation rate $f(\cdot)$ is a concave or convex increasing function
and $\psi(\cdot) \equiv 1$.
Various stochastic bounds and comparison results are established as well.

We leverage the various bounds and stochastic comparison results to
obtain the limiting distribution of the scaled queue length and delay
in heavy-traffic conditions.
The heavy-traffic asymptotics exhibit a sharp trichotomy.
The first heavy-traffic regime emerges in scenarios~(ii) and~(iii) described above. In this regime the scaled queue length
and delay converge to random variables with a gamma distribution.
The commonality between these two scenarios lies in the fact that the
ratio $f(i) / \psi(i)$ of the activation rate and the vacation
probability is linear in the queue length.
Loosely speaking, this means that the amount of vacation time is
inversely proportional to the queue length.
This proportionality property also holds for polling systems and in particular
vacation queues with so-called branching type service disciplines,
where the number of customers served in between two vacations
(switch-over periods) is proportional to the queue length at the start
of the service period.
Interestingly, the scaled queue length and delay for these types of
service disciplines have been proven to converge to random variables
with a gamma distribution in heavy-traffic conditions as well.
The significance of the ratio $f(i) / \psi(i)$ may also be recognized
when the activation rate and vacation probability are in fact constant,
i.e., independent of the queue length.
In that case, the server is active a fixed fraction of the time which only
depends on the activation rate and vacation probability through their ratio.

In the second heavy-traffic regime, which emerges in scenario~(i)
described above, the scaled queue length and delay both converge to
an exponentially distributed random variable with the same mean as in
the corresponding ordinary M/G/1 queue without any vacations.
In other words, the impact of the vacations completely vanishes in
heavy-traffic conditions.
Note that in this scenario the vacation probability falls off
{\em faster} than the inverse of the queue length.

The third heavy-traffic regime manifests itself when the vacation
probability decays {\em slower} than the inverse of the queue length, e.g.,
like the inverse of the queue length raised to a power less than one.
In that case, the queue length and delay, scaled by their respective means,
converge to one in distribution, while the mean values increase
an order-of-magnitude faster with the traffic intensity than in the
first two regimes.

While the above results are of independent interest from a queueing
perspective, they are also particularly relevant for the analysis
of distributed medium access control algorithms in wireless networks,
which in fact was the main motivation for the present work.
Emerging wireless networks typically lack any centralized control
entity for regulating transmissions, and instead vitally rely on
the individual nodes to operate autonomously and fairly share the
medium in a distributed fashion.
A particularly popular mechanism for distributed medium access control
is provided by the CSMA (Carrier-Sense Multiple-Access) protocol.
In the CSMA protocol each node attempts to access the medium after
a certain back-off time, but nodes that sense activity of interfering
nodes freeze their back-off timer until the medium is sensed idle.
From a queueing perspective, the back-off times may be interpreted
as vacation periods during which no transmissions take place,
even when packets may be queued up.

Despite their asynchronous and distributed nature, CSMA-like algorithms
have been shown to offer the capability of achieving the full capacity
region and thus match the optimal throughput performance of centralized
scheduling algorithms operating in slotted time \cite{JW08,JW10,LYPCP08}.
Based on this observation, various clever algorithms have been
developed for finding the back-off rates that yield specific target
throughput values or that optimize concave throughput utility functions
in scenarios with saturated buffers \cite{JW08,JW10,ME08}.

In the same spirit, several powerful algorithms have been devised for
adapting the back-off probabilities based on the {\em queue lengths} in
non-saturated scenarios \cite{JSSW10,RSS09,SST11}.
Roughly speaking, the latter algorithms provide maximum-stability
guarantees under the condition that the back-off probabilities of the
various nodes are reciprocal to the logarithms of the queue lengths.
Unfortunately, however, such back-off probabilities can induce
excessive queue lengths and delays, which has triggered a strong
interest in developing approaches for improving the delay performance
\cite{BBL11a,BBL11b,BBLP11,GS10,JLNSW11,LM10,NTS10,SS10}.
In particular, it has been shown that more aggressive schemes,
where the back-off probabilities decay faster to zero as function of
the queue lengths, can reduce the delays.
The heavy-traffic results described above offer a useful indication
of the impact of the choice of the back-off probabilities on the
delay performance.
It is worth observing that the vacation model does not account for
the effects of the network topology, and highly aggressive schemes
which are optimal in a single-node scenario, may in fact fail to
achieve maximum stability in certain types of topologies~\cite{GBW12}.
However, the single-node results provide fundamental insight how
the role of the back-off probabilities may inherently inflate
queue lengths and delays.

The remainder of the paper is organized as follows.
In Section~\ref{modeld} we present a detailed model description.
We provide an exact analysis of the model in Section~\ref{exactana},
which yields formulas for the stationary queue length distribution
in some specific cases.
In Section~\ref{secbounds} we derive lower and upper bounds for the
mean queue length and we establish a stochastic relation between
systems with different functions $\DP(\cdot)$ and $f(\cdot)$.
We study heavy-traffic behavior in Section~\ref{HTresults}
and identify three qualitatively different regimes.
In Section~\ref{concl} we summarize our findings and discuss the implications for wireless networks.
Finally, Appendix~\ref{appendix} contains some proofs that have
been relegated from the main text.

\section{Model description}
\label{modeld}
We consider an M/G/1 queue with vacations. That is, we consider a queueing system with one server that can be active or inactive. Customers arrive according to a Poisson process with rate~$\lambda$, independent of the state of the system.
Let $\sigma(t)$ indicate whether the server is active at time $t$ ($\sigma(t)=1$) or not $(\sigma(t)=0$) and denote by $\TP(t)$ the number of customers in the system at time~$t$.
When inactive, no customer is served and we say that the server is on vacation. The server becomes active after some time that may depend on the number of waiting customers at the beginning of the vacation period and the number of customers that arrive during the vacation period, but it may not depend on future arrivals.
Denote by $\AP(i,m)$ the probability that exactly $m$ customers arrive during a vacation period that begins with $i$ customers in the system, where $\AP:[0,\infty)\times [0,\infty) \mapsto [0,1]$. Further we assume $\AP(0,0)=0$, i.e.~the server does not activate if no customers are present in the system. Let the random variable $X_i$ denote the number of arrivals during a vacation period that begins with $i$ customers in the system. We will assume that $\AP(\cdot)$ is such that $\expect{X_i^2}<\infty$.
When active, customers are served and the service times, generically denoted by~$B$, are generally distributed with distribution function $F_B(\cdot)$ and Laplace-Stieltjes transform $\tilde{B}(\cdot)$. We assume that $\expect{B^2}<\infty$ and that the service times are independent of the arrival and vacation times. Right after a service completion that leaves $i$ customers behind, the server becomes inactive with probability $\DP(i)$, where $\DP:[0,\infty)\mapsto [0,1]$. Further we assume $\DP(0)=1$, i.e.~the server always becomes inactive if no customers are left in the system.

Let $\rho=\lambda\expect{B}$ denote the traffic intensity of the system. Throughout this paper, we denote the generating function of a non-negative and discrete random variable~$W$ by $G_W(r)=\expect{r^W}$, with $r\in [0,1]$. Note that
\[
G_{X_i}(r)= \sum_{m=0}^{\infty}\AP(i,m)r^m.
\]

Let $W_1 \eqd W_2$ denote that two random variables $W_1$ and $W_2$ are equal in distribution, so that $\pr{W_1\leq w} = \pr{W_2\leq w}$ for all $w$. Further, let $W_1 \stl W_2$ denote that $W_1$ is stochastically larger than $W_2$, so that $\pr{W_1\leq w} \leq \pr{W_2\leq w}$ for all $w$. Finally, let $W_1 \stsl W_2$ denote that $W_1$ is stochastically strictly larger than $W_2$, so that $W_1 \stl W_2$ and, additionally, $\pr{W_1\leq w} < \pr{W_2\leq w}$ for some $w$.

\section{Exact analysis}
\label{exactana}
Denote by $Z_n$ the number of customers just after the $n$-th service completion and by $A_n$ the number of arrivals during the $n$-th service. Then $(Z_n)_{n\in \mathbb{N}_0}$ constitutes a Markov chain with transition probabilities
\[
\pr{Z_{n+1}=j|Z_n=i} = (1-\DP(i))\pr{A_n=j-i+1}+\DP(i)\pr{X_i+A_n=j-i+1},
\]
for $j\geq i-1$ and
\[
\pr{Z_{n+1}=j|Z_n=i}=0,
\]
for $j<i-1$. Because $X_i$ and $A_n$ are assumed to be independent,
\begin{align}
\expect{r^{Z_{n+1}}|Z_n=i}&=(1-\DP(i))r^{i-1}G_{A_n}(r)+\DP(i)r^{i-1}G_{X_i}(r)G_{A_n}(r) \nonumber \\
&=r^{i-1}G_{A_n}(r)(1+\DP(i)(G_{X_i}(r)-1)), \label{eq1}
\end{align}
where $G_{A_n}(r)=\tilde{B}(\lambda(1-r))$ for all $n$. Using this relation we can find a sufficient condition for stability of the system.
\begin{lemma}
\label{stable}
The Markov chain $(Z_n)_{n \in \mathbb{N}_0}$ is positive recurrent if
\begin{equation}
\label{stabcond}
\limsup_{i\rightarrow\infty} \DP(i)\expect{X_i}<1-\rho.
\end{equation}
\end{lemma}
\begin{proof}
This result is proved in~\cite{HM88}, using the results in~\cite{Crabill68}. For a short proof note that from~\eqref{eq1} we find $\expect{Z_{n+1}|Z_n=i}=i-1+\rho+\DP(i)\expect{X_i}$. The result now follows immediately from Pakes' Lemma~\cite{Pak69} as $\expect{X_i}<\infty$ for all $i\geq 0$ by assumption.
\end{proof}
In words, Lemma~\ref{stable} states that for stability it is sufficient that the system is busy serving customers more than a fraction $\rho$ of the time if the number of customers in the system is large.

We henceforth assume the system is stable, i.e.~$\DP(\cdot)$ and $\AP(\cdot)$ are such that the condition in~\eqref{stabcond} is satisfied. Let the random variable~$Z$ have the stationary distribution of the embedded Markov chain $(Z_n)_{n\in \mathbb{N}_0}$, i.e.
\[
\pr{Z=j}=\lim_{n\rightarrow\infty}\pr{Z_n=j|Z_0=k},\quad k\geq 0.
\]
By the PASTA property and a level crossings argument we know that $\{\pr{Z=j},j\geq 0\}$ is also the stationary distribution of the number of customers in the system~$\TP$, with
\[
\pr{\TP =j}=\lim_{t\rightarrow\infty}\pr{\TP(t)=j|\TP(0)=k},
\]
for any $k\geq 0$.
Hence, using~\eqref{eq1}, we obtain the relation
\[
G_{\TP}(r)=\frac{1}{r} \tilde{B}(\lambda(1-r))\Big(G_{\TP}(r) + \sum_{i=0}^{\infty}\DP(i)r^i\pr{\TP=i}(G_{X_i}(r)-1)\Big),
\]
which corresponds to~\cite[Eq.~(2)]{HM88}. Equivalently,
\begin{equation}
\label{genG}
G_{\TP}(r)=\frac{\tilde{B}(\lambda(1-r))\sum_{i=0}^{\infty}\DP(i)r^i\pr{\TP=i}(1-G_{X_i}(r))}{\tilde{B}(\lambda(1-r))-r}.
\end{equation}

\begin{example}
\label{PKexmp}
One activation scheme would be to never de-activate when the system is nonempty right after a service completion, i.e.,~$\DP(i)=0$ for $i\geq 1$. Similarly we could say that the server always activates immediately if there are waiting customers at the beginning of the vacation period, i.e.,~$\AP(i,0)=1$ for $i\geq 1$, and $\AP(i,m)=0$ otherwise, so that $G_{X_i}(r)=1$ for $i\geq 1$. For this activation scheme~\eqref{genG} simplifies to
\[
G_{\TP}(r)=\frac{\tilde{B}(\lambda(1-r))\pr{\TP=0}(1-G_{X_0}(r))}{\tilde{B}(\lambda(1-r))-r}.
\]
Using that $G_{\TP}(1)=1$ and applying l'H\^opital's rule yields
\[
\pr{\TP=0}=\frac{1-\rho}{\expect{X_0}},
\]
and hence
\begin{equation}
\label{PKform}
G_{\TP}(r)=\frac{(1-\rho)\tilde{B}(\lambda(1-r))(1-G_{X_0}(r))}{\expect{X_0}(\tilde{B}(\lambda(1-r))-r)}.
\end{equation}
Note that if the server waits for exactly one customer to arrive if there are no waiting customers at the beginning of the vacation period, i.e.~$X_0\equiv 1$, then~\eqref{PKform} becomes the classical Pollaczek-Khinchin formula for the standard M/G/1 queue without vacations,
\begin{equation}
\label{PKform2}
G_{\TP}(r)=G_{{\TP}_{{\rm M/G/1}}}(r)=\frac{(1-\rho)\tilde{B}(\lambda(1-r))(1-r)}{\tilde{B}(\lambda(1-r))-r}.
\end{equation}
\end{example}
The Fuhrmann-Cooper decomposition~\cite{FC85} relates $G_{\TP}(r)$ to the Pollaczek-Khinchin formula through
\begin{equation}
\label{fuco1}
G_{\TP}(r) = G_{{\TP}_{{\rm M/G/1}}}(r) G_{{\TP}_{I}}(r),
\end{equation}
where ${\TP}_I$ denotes the number of customers in the system at an arbitrary epoch during a non-serving (vacation) period. This decomposition property can be derived from~\eqref{genG}. For this denote by ${\TP}_{\rm begin}$ and ${\TP}_{\rm end}$ the number of customers in the system at, respectively, the beginning and the end of a vacation period, and let $\gamma$ be the probability that the server becomes inactive after a departure,
\[
\gamma = \sum_{i=0}^{\infty} \DP(i) \pr{\TP=i}.
\]
Because the system is stable the expected number of arrivals between two service completions is equal to the expected number of service completions, which equals one. Therefore,
\[
\rho + \gamma (\expect{{\TP}_{\rm end}}-\expect{{\TP}_{\rm begin}}) = 1,
\]
and the expected number of arrivals during a vacation period is therefore given by
\[
\expect{{\TP}_{\rm end}}-\expect{{\TP}_{\rm begin}}= \frac{1-\rho}{\gamma}.
\]
Further note that
\[
\pr{{\TP}_{\rm begin}=i} = \frac{1}{\gamma} \pr{L=i}\DP(i).
\]
From~\eqref{genG} we now find
\begin{align*}
G_{\TP}(r)&= \frac{(1-\rho)\tilde{B}(\lambda(1-r))(1-r)}{\tilde{B}(\lambda(1-r))-r} \frac{\sum_{i=0}^{\infty}\DP(i)r^i\pr{\TP=i}(1-G_{X_i}(r))}{(1-\rho)(1-r)}\\
&= G_{{\TP}_{{\rm M/G/1}}}(r) \frac{\sum_{i=0}^{\infty}\frac{1}{\gamma}\DP(i)r^i\pr{\TP=i}(1-G_{X_i}(r))}{\frac{1}{\gamma}(1-\rho)(1-r)}\\
&= G_{{\TP}_{{\rm M/G/1}}}(r) \frac{\sum_{i=0}^{\infty}r^i\pr{{\TP}_{\rm begin}=i}(1-G_{X_i}(r))}{(1-r)(\expect{{\TP}_{\rm end}}-\expect{{\TP}_{\rm begin}})}\\
&= G_{{\TP}_{{\rm M/G/1}}}(r) \frac{G_{{\TP}_{\rm begin}}(r) - G_{{\TP}_{\rm end}}(r)}{(1-r)(\expect{{\TP}_{\rm end}}-\expect{{\TP}_{\rm begin}})},
\end{align*}
yielding~\eqref{fuco1}, see~\cite{Borst95}. Thus to find $G_{\TP}(r)$ we can either solve equation~\eqref{genG} or find $G_{{\TP}_{I}}(r)$ and then use the Fuhrmann-Cooper decomposition.

In the remainder of this section we will analyze the system for several choices of $\AP(\cdot)$ and $\DP(\cdot)$.

\subsection{Equal vacation distributions}
\label{equalvac}
In this subsection we assume that $X_i \eqd X$ for $i\geq 1$, with $X$ some generic random variable. We further assume that $X\stsl 0$, so that with nonzero probability at least one customer arrives during any vacation. The case $X \eqd 0$ is already solved in Example~\ref{PKexmp}. Next, if a vacation starts with no customers in the system, we assume that first $X$ customers arrive. After this, if the system is still empty, the vacation is extended in an arbitrary way until at least one customer has arrived. We thus have $X_0 \eqd X$ if $X\stl 1$, i.e.~if $\pr{X=0}=0$, and $X_0 \stsl X$ otherwise, i.e.~if $\pr{X=0}>0$.

To summarize, in this subsection we study the following scenario.
\begin{scenario}
\label{scen1}
$X_i\eqd X \stsl 0$ for all $i\geq 1$ and either $X_0 \eqd X$ and $\pr{X=0}=0$ or $X_0 \stsl X$ and $\pr{X=0}>0$.
\end{scenario}
Note that in this scenario we have $G_{X_i}(r)=G_{X}(r)$ for all $i\geq 0$ and all $r \in [0,1]$ if $G_{X}(0)=0$ and $G_{X_0}(r)<G_{X_i}(r)=G_X(r)$ for all $i\geq 1$ and all $r \in [0,1)$ if $G_{X}(0)>0$. So from~\eqref{genG} and $\DP(0)=1$ it follows that
\begin{equation}
\label{geng}
G_{\TP}(r)=\frac{\tilde{B}\Big(\lambda(1-r))(\pr{\TP=0}(G_X(r)-G_{X_0}(r))+(1-G_X(r))\sum_{i=0}^{\infty}\DP(i)r^i\pr{\TP=i}\Big)}{\tilde{B}(\lambda(1-r))-r}.
\end{equation}

Equation~\eqref{geng} seems hard to solve in general, but we are able to find solutions for several specific choices for $\DP(\cdot)$. Before analyzing~\eqref{geng} in more detail we now first give a prototypical example of a system that belongs to Scenario~\ref{scen1}. This example describes a back-off mechanism used in wireless networks.
\begin{example}
\label{vactime}
Consider a server that always waits for a certain time $V$, independent of the arrivals during this time. After this time the server activates if there are customers present in the system, and otherwise the server again waits for a time $V$ (independent of the previous time) and repeats this procedure until there are customers present in the system. Assume $V$ is generally distributed with distribution function $F_V(\cdot)$ and Laplace-Stieltjes transform $\tilde{V}(\cdot)$. Denoting by $\alpha_m$ the probability that exactly $m$ customers arrive during a time $V$,
\[
\alpha_m = \int_{0}^{\infty} \frac{(\lambda t)^m}{m!}\e^{-\lambda t} dF_V(t),
\]
we have $\AP(i,m)=\alpha_m$, for $i\geq 1$, and $\AP(0,m)=\alpha_m/(1-\alpha_0)$, for $m\geq 1$.
Further, we get
\[
G_{X_0}(r) = \sum_{m=1}^{\infty} \AP(0,m) r^m = \frac{1}{1-\alpha_0} \int_{t=0}^{\infty} \sum_{m=1}^{\infty} \frac{(\lambda t r)^m}{m!}\e^{-\lambda t} dF_V(t) = \frac{\tilde{V}(\lambda(1-r))-\tilde{V}(\lambda)}{1-\tilde{V}(\lambda)},
\]
where the interchange of summation and integration is justified by Beppo Levi's theorem, see e.g.~\cite{CK04}. Similarly, we get $G_{X_i}(r)=\tilde{V}(\lambda(1-r))$ for $i\geq 1$.

Note that if $V$ is exponentially distributed with mean $1/\nu$ we find $X_i \eqd X$, $i \geq 1$, where $X$ is a geometric random variable, with
\begin{equation}
\label{Pexp}
G_{X}(r)=\frac{\nu}{\lambda(1-r)+\nu}.
\end{equation}
Further, $G_{X_0}(r)=rG_{X_1}(r)$ as $X_0 \eqd X_1 + 1$ in this case.
\end{example}

We will now use~\eqref{geng} to find $G_{\TP}(\cdot)$ if $\DP(i)=a^i$ or $\DP(i)=1/(i+1)$, two functions that we will need in the heavy-traffic analysis of the system, see Section~\ref{HTresults}. For this purpose first introduce
\begin{equation}
\label{defK}
K(r)=\frac{\tilde{B}(\lambda(1-r))(G_X(r)-G_{X_0}(r))}{\tilde{B}(\lambda(1-r))-r},
\end{equation}
with $K(r)\equiv 0$ if $X_0 \eqd X$. Define
\begin{equation}
\label{defY}
Y(r)=\frac{\tilde{B}(\lambda(1-r))(1-G_{X}(r))}{\tilde{B}(\lambda(1-r))-r}
\end{equation}
and note that an alternative expression for $Y(\cdot)$ is given by
\begin{equation}
\label{defYalt}
Y(r)=G_{\TP_{{\rm M/G/1}}}(r) G_{X^{\rm {res}}}(r) \frac{\expect{X}}{1-\rho},
\end{equation}
with $G_{\TP_{{\rm M/G/1}}}(r)$ the generating function of the number of customers in a standard M/G/1 queue without vacations~\eqref{PKform2}, and $G_{X^{\rm {res}}}(r)$ the generating function of the number of arrivals in a residual vacation period,
\[
\expect{r^{X^{\rm {res}}}} = \frac{1-G_{X}(r)}{(1-r)\expect{X}}.
\]
Similarly, we can write
\begin{align}
K(r)&=G_{\TP_{{\rm M/G/1}}}(r) \frac{G_X(r)-G_{X_0}(r)}{(1-\rho)(1-r)} \nonumber \\
&= G_{\TP_{{\rm M/G/1}}}(r) \Big( \expect{X_0} G_{X_0^{\rm {res}}}(r) - \expect{X} G_{X^{\rm {res}}}(r) \Big) \frac{1}{1-\rho}. \label{defKalt}
\end{align}
Further, by l'H\^opital's rule,
\[
Y(1)=\lim_{r\uparrow 1} Y(r) = \frac{\expect{X}}{1-\rho}
\]
and
\[
K(1)=\lim_{r\uparrow 1} K(r) =\frac{\expect{X_0}-\expect{X}}{1-\rho}.
\]
Note that $Y(1)>0$ as $\expect{X}>0$ and $\rho<1$ for stability. Also note that $K(1)>0$ if $X_0 \stsl X$.

Finally note that the generating function $G_W(r)$ of any non-negative discrete random variable $W$ is a non-negative continuously differentiable function on $[0,1]$, as follows from the definition of a generating function. Hence $Y(r)$ and $K(r)$ are non-negative continuously differentiable functions on $[0,1]$.

\begin{theorem}
\label{eqai}
For {\rm Scenario~\ref{scen1}} and $\DP(i)=a^i$ with $0\leq a <1$, $i\geq 0$,

${\rm (i)}$ If $X_0 \eqd X$,
\begin{equation}
\label{eqaieq1}
G_{\TP}(r)=\frac{\prod_{i=0}^{\infty}Y(a^ir)}{\prod_{i=0}^{\infty}Y(a^i)}.
\end{equation}

${\rm (ii)}$ If $X_0 >_{{\rm st}} X$,
\begin{equation}
\label{eqaieq2}
G_{\TP}(r)=\frac{\sum_{j=0}^{\infty}K(a^jr)\prod_{i=0}^{j-1}Y(a^ir)}{\sum_{j=0}^{\infty}K(a^j)\prod_{i=0}^{j-1}Y(a^i)},
\end{equation}
with $\prod_{i=0}^{-1} g(i)= 1$ for any function $g(\cdot)$.
\end{theorem}
\begin{proof}
From Lemma~\ref{stable} we obtain that the system is stable if $\rho<1$, as $0\leq a < 1$. We will now first prove the result for case~${\rm (i)}$, for which~\eqref{geng} simplifies to
\begin{equation}
\label{ag}
G_{\TP}(r)=\frac{\tilde{B}(\lambda(1-r))(1-G_{X}(r))G_{\TP}(ar)}{\tilde{B}(\lambda(1-r))-r}=Y(r)G_{\TP}(ar).
\end{equation}
Upon iteration this gives, using $G_{\TP}(0)=\pr{\TP=0}$,
\begin{equation}
\label{agfin}
G_{\TP}(r)=\pr{\TP=0}\prod_{i=0}^{\infty}Y(a^ir).
\end{equation}
Finally, using $G_{\TP}(1)=1$, we obtain
\begin{equation}
\label{prob01}
\pr{\TP=0}=\frac{1}{\prod_{i=0}^{\infty}Y(a^i)}.
\end{equation}
Combining~\eqref{agfin} and~\eqref{prob01} yields assertion~\eqref{eqaieq1}. In Lemma~\ref{convproofs} we prove that $\prod_{i=0}^{\infty}Y(a^ir)$ converges for all $r\in [0,1]$, so in particular $\pr{\TP=0}>0$.

For case~${\rm (ii)}$ equation~\eqref{geng} simplifies to
\begin{eqnarray*}
G_{\TP}(r)&=&\frac{\tilde{B}(\lambda(1-r))\Big(\pr{\TP=0}(G_{X}(r)-G_{X_0}(r))+(1-G_{X}(r))G_{\TP}(ar)\Big)}{\tilde{B}(\lambda(1-r))-r}\\
&=& K(r)\pr{\TP=0} + Y(r)G_{\TP}(ar).
\end{eqnarray*}
Iterating this and using $G_{\TP}(0)=\pr{\TP=0}$ we get
\[
G_{\TP}(r)=\pr{\TP=0}\Big(\sum_{j=0}^{\infty}K(a^jr)\prod_{i=0}^{j-1}Y(a^ir) + \prod_{i=0}^{\infty} Y(a^ir) \Big).
\]
Now note that $Y(0)=1-G_X(0)$, so $Y(0)<1$ by assumption. Thus, as $Y(\cdot)$ is continuous and $0\leq a <1$,
\begin{equation}
\label{agfin2}
G_{\TP}(r)=\pr{\TP=0}\sum_{j=0}^{\infty}K(a^jr)\prod_{i=0}^{j-1}Y(a^ir).
\end{equation}

From~\eqref{agfin2} and $G_{\TP}(1)=1$ we get~\eqref{eqaieq2}. In Lemma~\ref{convproofs} we prove that $\sum_{j=0}^{\infty}K(a^jr)\prod_{i=0}^{j-1}Y(a^ir)$ converges for all $r\in [0,1]$, so in particular
\[
\pr{\TP=0}=\frac{1}{\sum_{j=0}^{\infty}K(a^j)\prod_{i=0}^{j-1}Y(a^i)} >0,
\]
which completes the proof.
\end{proof}

\begin{theorem}
\label{eq1over}
For {\rm Scenario~\ref{scen1}} and $\DP(i)=1/(i+1)$, $i\geq 0$, with $\alpha(r)=\int_{r}^1\frac{Y(x)}{x}dx$,

${\rm (i)}$ If $X_0 \eqd X$,
\begin{equation}
\label{eq1overeq1}
G_{\TP}(r)=\frac{(1-\rho)Y(r)}{r\expect{X}}\e^{-\alpha(r)}.
\end{equation}

${\rm (ii)}$ If $X_0 >_{{\rm st}} X$,
\begin{equation}
\label{eq1overeq2}
G_{\TP}(r)=\pr{\TP=0}\Big(K(r)+\frac{Y(r)}{r}\e^{-\alpha(r)}\int_0^r K(y)\e^{\alpha(y)}dy\Big),
\end{equation}
with
\begin{equation}
\label{eq1overeq2L0}
\pr{\TP=0}=\frac{1}{K(1)+Y(1)\int_0^1K(x)\e^{\alpha(x)} dx}.
\end{equation}
\end{theorem}
\begin{proof}
From Lemma~\ref{stable} it follows that the system is stable if $\rho<1$ as $\lim_{i\rightarrow\infty}\DP(i)=0$.
For case~${\rm (i)}$, equation~\eqref{geng} gives
\[
G_{\TP}(r)=\frac{\tilde{B}(\lambda(1-r))(1-G_{X}(r))\sum_{i=0}^{\infty}\frac{1}{i+1}r^i\pr{\TP=i}}{\tilde{B}(\lambda(1-r))-r}.
\]
Now define
\[
H_{\TP}(r)=\sum_{i=0}^{\infty} \frac{1}{i+1} r^{i+1} \pr{\TP=i},
\]
and note that $H'_\TP(r)=G_{\TP}(r)$. Thus, $H_{\TP}(r)$ can be found by solving the differential equation
\begin{equation}
\label{difeq1over}
H'_{\TP}(r)=\frac{Y(r) H_\TP(r)}{r}.
\end{equation}
We thus get
\[
H_\TP(r)=C \cdot {\rm{exp}}\Big(-\int_{r}^1\frac{Y(x)}{x}dx\Big)
\]
and
\[
G_{\TP}(r)=C \cdot \frac{Y(r)}{r}{\rm{exp}}\Big(-\int_{r}^1\frac{Y(x)}{x}dx\Big),
\]
where $C$ is some constant. Using $G_{\TP}(1)=1$ this gives~\eqref{eq1overeq1}.

For case~${\rm (ii)}$ we can find $H_{\TP}(r)$ by solving the differential equation
\begin{equation}
\label{difeq1over2}
H'_{\TP}(r)=\pr{\TP=0}K(r)+\frac{Y(r)H_{\TP}(r)}{r}.
\end{equation}
This gives
\[
H_{\TP}(r)=\e^{-\alpha(r)}\Big(\pr{\TP=0}\int_{0}^r K(y)\e^{\alpha(y)}dy+C\Big)
\]
and
\begin{equation}
\label{GZlin}
G_{\TP}(r)=\pr{\TP=0}K(r)+\frac{Y(r)}{r}\e^{-\alpha(r)}\Big(\pr{\TP=0}\int_0^{r} K(y)\e^{\alpha(y)}dy+C\Big),
\end{equation}
for some constant $C$. As $G_{\TP}(r)$ is a generating function, boundary conditions are given by $G_{\TP}(0)=\pr{\TP=0}$ and $G_{\TP}(1)=1$.

To solve this boundary problem, note that, using integration by parts,
\[
\alpha(r) = -\log(r)Y(r) - \int_{r}^1 Y'(x)\log(x)dx.
\]
Further $Y(0)=1-G_{X}(0)$, so $Y(0)<1$ in this case, and for all $x \in [0,1]$, $Y'(x)\leq Y'(1) <\infty$ as $\expect{B^2}<\infty$ and $\expect{X^2}<\infty$. Therefore,
\[
\frac{Y(r)}{r}\e^{-\alpha(r)} = Y(r)r^{Y(r)-1} {\rm exp}\Big(\int_{r}^{1}Y'(x)\log(x)dx\Big)
\]
and thus
\[
\lim_{r\downarrow 0} \frac{Y(r)}{r}\e^{-\alpha(r)} \geq \lim_{r\downarrow 0}  Y(r){\rm exp}\Big((Y(r)-1)\log(r)\Big) {\rm exp}(-Y'(1)) = \infty.
\]
Hence, to get $G_{\TP}(0)=\pr{\TP=0}$, we need to set $C=0$, so that~\eqref{GZlin} becomes
\[
G_{\TP}(r)=\pr{\TP=0}\Big(K(r)+\frac{Y(r)}{r}\e^{-\alpha(r)}\int_0^{r} K(y)\e^{\alpha(y)}dy\Big).
\]
Finally, using $G_{\TP}(1)=1$ we find~\eqref{eq1overeq2}.
\end{proof}

Equation~\eqref{geng} can be used to find $G_{\TP}(\cdot)$ for other functions~$\DP(\cdot)$ as well. For example, if~$\DP(i)=g(i)$ for $i<I$ and $\DP(i)=a^i$ for $i \geq I$, for some function $g(\cdot)$, $0\leq g(\cdot) \leq 1$, and some $I\in \mathbb{N}_0$, one can use the same approach as used in the proof of Theorem~\ref{eqai} to find $G_{\TP}(\cdot)$ in terms of $\pr{L=j}$, $j=0,\dots,I-1$. Although interesting, it is beyond the scope of this paper to analyze this in more detail.

\subsection{State-dependent vacation lengths}
\label{nonhomsec}
In this subsection we consider a system that has state-dependent activation rules. More precisely, we assume that the server becomes active at the first jump of a non-homogeneous Poisson process with rate $f(i)$ when $L(t)=i$. This gives the following scenario.
\begin{scenario}
\label{scen2}
\[
\AP(i,m)=\frac{f(i+m)}{\lambda+f(i+m)}\prod_{j=0}^{m-1} \frac{\lambda}{\lambda +f(i+j)}, \quad i,m\geq 0,
\]
where $f:[0,\infty) \mapsto [0,\infty)$ and $f(0)=0$.
\end{scenario}
Note that we need $f(0)=0$ as $\AP(0,0)=0$ if and only if $f(0)=0$.
\begin{theorem}
\label{nonhomlin}
For {\rm Scenario~\ref{scen2}} with $\DP(i)=1$ and $f(i)=\nu i$, $i\geq 0$,
\begin{equation}
\label{nonhomlineq}
G_{\TP}(r)=\frac{(1 - \rho)\tilde{B}(\lambda(1-r))(1-r)}{\tilde{B}(\lambda(1-r))-r}{\rm exp}\Big({\int_r^1 \frac{-\lambda(1-x)}{\nu (\tilde{B}(\lambda(1-x))-x)}dx}\Big).
\end{equation}
\end{theorem}
\begin{proof}
First note that, for all $m>0$, $\AP(i,m) \rightarrow 0$ as $i \rightarrow \infty$. So $\expect{X_i} \rightarrow 0$ and the system is stable if $\rho < 1$, see Lemma~\ref{stable}.

To prove this theorem we will first determine $G_{\TP_I}(r)$, and then $G_{\TP}(r)$ follows from the Fuhrmann-Cooper decomposition~\eqref{fuco1}.

In order to determine the distribution of $\TP_I$ we cut out all the services and replace these by instantaneous jumps whose sizes are the number of arrivals $\TP_A$ during an arbitrary service time, with $G_{\TP_A}(r)=\tilde{B}(\lambda(1-r))$. These jumps occur at rate $\nu i$ when there are $i$~customers in the system. Thus, as the function $f(\cdot)$ is linear, the distribution of $\TP_I$ corresponds to that of a continuous-time branching process with immigration: Particles arrive as a Poisson process with rate~$\lambda$ and each particle is independently and instantaneously replaced at rate~$\nu$ by a new group of $\TP_A$ particles. Branching processes of this type were studied by Sevast'yanov~\cite{Sev57} and applying \cite[Theorem 1]{Sev57} to our situation yields
\[
G_{\TP_I}(r)={\rm exp}\Big({\int_r^1 \frac{-\lambda(1-x)}{\nu (\tilde{B}(\lambda(1-x))-x)}dx}\Big).
\]
By~\eqref{fuco1} we then obtain~\eqref{nonhomlineq}.
\end{proof}

\subsubsection{Exponentially distributed service times}
For generally distributed service times it is difficult to interpret Theorem~\ref{nonhomlin}, but for exponentially distributed service times we can using the following result.
\begin{corollary}
\label{nonhomlinexp}
For {\rm Scenario~\ref{scen2}} with $\DP(i)=1$ and $f(i)=\nu i$, $i \geq 0$, and exponentially distributed service times with mean $1/\mu$,
\begin{equation}
\label{nonhomlinexpeq}
G_{\TP}(r)=\left(\frac{1-\rho}{1-\rho r}\right)^{1+\lambda/\nu} \e^{(r-1)\lambda/\nu}.
\end{equation}
\end{corollary}
\begin{proof}
Evaluating the integral in Theorem~\ref{nonhomlin} with $\tilde{B}(s)=\frac{\mu}{\mu+s}$ gives~\eqref{nonhomlinexpeq}.
\end{proof}
Notice that~\eqref{nonhomlinexpeq} is the product of two generating functions, so that the distribution of~$\TP$ is a convolution of a negative binomial distribution and a Poisson distribution.

For exponentially distributed service times, and any choice of $f(\cdot)$ and $\DP(\cdot)$, $(\TP(t),\sigma(t))_{t\geq 0}$ is a continuous-time Markov process with state space $\{0,1,2,\dots\} \times \{0,1\}$ and state $(i,j)$ representing $i$ customers in the system and $j=0$ when the server is inactive and $j=1$ when the server is active. Transitions from $(i,0)$ to $(i+1,0)$ and from $(i,1)$ to $(i + 1,1)$ occur at rate~$\lambda$, corresponding to an arrival of a customer, and transitions from $(i,0)$ to $(i,1)$ occur at rate~$f(i)$, corresponding to a server activation. Further, transitions from $(i + 1,1)$ to $(i,0)$ occur at rate~$\mu \DP(i)$, corresponding to a service completion and server de-activation. Finally, transitions from $(i + 1,1)$ to $(i,1)$ occur at rate $\mu(1- \DP(i))$, corresponding to a service completion without server de-activation. With $\stat{i}{k}$ the stationary probability that the Markov process resides in state $(i,k)$, we have the balance equations
\begin{eqnarray*}
\lambda \stat{0}{0} &=& \mu\stat{1}{1},\\
(\lambda+\mu) \stat{1}{1} &=& f(1)\stat{1}{0} + \mu(1-\DP(1))\stat{2}{1},\\
(\lambda + f(i))\stat{i}{0} &=& \lambda\stat{i-1}{0} + \mu \DP(i)\stat{i+1}{1}, \quad i\geq 1,\\
(\lambda + \mu)\stat{i}{1} &=& \lambda\stat{i-1}{1}+f(i)\stat{i}{0}+\mu(1-\DP(i))\stat{i+1}{1}, i\geq 2.
\end{eqnarray*}

This set of balance equations can be solved for several choices of $f(\cdot)$ and $\DP(\cdot)$. For example, $f(i)=b^i$, with $b>1$, and $\DP(i)=1$, $i\geq 0$, yields a result similar to the result of Theorem~\ref{eqai}. Also, the result of Corollary~\ref{nonhomlinexp} can be derived in this way.

The next theorem gives a class of functions for which the distribution of the total number of customers in the system in steady state is negative binomial.
\begin{theorem}
\label{negativbin}
For {\rm Scenario~\ref{scen2}} with $\DP(i)=k/(k+i)$, $i\geq 1$, and $f(i)=\mu i/(i+k-1)$, $i \geq 0$, with $k\geq 0$, and exponentially distributed service times with mean $1/\mu$,
\begin{equation}
\label{negativbineq}
G_{\TP}(r)= \left(\frac{1-\rho}{1-\rho r}\right)^{k+1}.
\end{equation}
\end{theorem}
\begin{proof}
It can be checked that
\[
\stat{i}{0}=\binom{i+k-1}{i}(1-\rho)^{k+1}\rho^i,
\]
and
\[
\stat{i}{1}=\binom{i+k-1}{i-1}(1-\rho)^{k+1}\rho^i,
\]
solve the set of balance equations and the normalization equation $\sum_{i,j} \stat{i}{j}=1$. Thus $\pr{L=0}=\stat{0}{0}=(1-\rho)^{k+1}$ and for $i\geq 1$,
\[
\pr{L=i}=\stat{i}{0}+\stat{i}{1}=\binom{i+k}{i}(1-\rho)^{k+1}\rho^i,
\]
which gives~\eqref{negativbineq}.
\end{proof}

Note that the functions in Theorem~\ref{negativbin} describe an M/M/1 queue if $k=0$, as one always has an immediate transition from $(1,0)$ to $(1,1)$ and there are no transitions from $(i,1)$ to $(k,0)$ for any $k \geq 0$ if $i\geq 2$.

Further, $k=1$ leads to a special case of Theorem~\ref{eq1over}, because $\DP(i)=1/(i+1)$, $i \geq 1$, the service times are exponentially distributed with mean $1/\mu$ and the vacation discipline of Example~\ref{vactime} is used with the vacation time distribution identical to the service time distribution.

The results of Corollary~\ref{nonhomlinexp} and Theorem~\ref{negativbin} could also be derived using a probabilistic approach. For the situation of Corollary~\ref{nonhomlinexp} the number of customers at an arbitrary epoch during a vacation period~$\TP_I$ can be related to the customers in a network of infinite-server queues with phase-type service requirement distributions.

For the situation of Theorem~\ref{negativbin} the vacation model behaves as an M/M/1 queue with $k$ permanent customers and a Random-Order-of-Service (ROS) discipline. The ROS discipline selects the next customer for service at random from those which were in the queue just before the service completion, and excludes a permanent customer whose service may just have been completed.

\section{Bounds}
\label{secbounds}
In Section~\ref{exactana} we derived exact results for several choices of $\AP(\cdot)$ and $\DP(\cdot)$. In this section we will derive bounds for the mean number of customers in the system. These bounds will be used in the heavy-traffic analysis in Section~\ref{HTresults}.
\subsection{Equal vacation distributions}
In this subsection we consider the class of vacation disciplines of Scenario~\ref{scen1} as described in Subsection~\ref{equalvac}. For this class we find the following lower bound.
\begin{theorem}
\label{boundg}
For {\rm Scenario~\ref{scen1}} and when $\DP(\cdot)$ is a strictly decreasing convex function,
\[
\expect{\TP} \geq \max\Big\{\DP^{-1}\Big(\frac{1-\rho}{\expect{X}}\Big), \frac{\lambda^2\expect{B^2}}{2(1-\rho)}+\rho \Big\}.
\]
\end{theorem}
\begin{proof}
In steady state, the mean number of activations per unit of time equals the mean number of de-activations per unit of time, so that
\begin{equation}
\label{actisdeactA}
\pr{\sigma=1}\frac{1}{\expect{B}}\expect{\DP(Z)}=\frac{\lambda}{\expect{X}}\pr{\sigma=0 \cap L>0},
\end{equation}
where $\sigma$ denotes the random variable with the steady-state distribution of the state of the server, i.e.,
\[
\pr{\sigma=j}=\lim_{t\rightarrow\infty}\pr{\sigma(t)=j|\sigma(0)=k},
\]
and $Z$ denotes, as before, the steady-state number of customers in the system right after service completions. Because $Z \eqd \TP$ and $\pr{\sigma=1}=\rho$,
\[
\lambda\expect{\DP(\TP)} \leq \frac{\lambda}{\expect{X}} \pr{\sigma=0} = \frac{\lambda}{\expect{X}} (1-\rho).
\]
Further, it follows from Jensen's inequality that, as $\DP(\cdot)$ is convex,
\[
\expect{\DP(\TP)} \geq \DP(\expect{\TP}).
\]
Since $\DP(\cdot)$ is decreasing we then get
\begin{equation}
\label{glowb}
\expect{L} \geq \DP^{-1}\Big(\frac{1-\rho}{\expect{X}}\Big).
\end{equation}
Finally, the Fuhrmann-Cooper decomposition~\eqref{fuco1} implies
\[
\expect{\TP} = \frac{\lambda^2 \expect{B^2}}{2(1 - \rho)} + \rho + \expect{\TP_{I}},
\]
where $\TP_I$ denotes the number of customers during a non-serving (vacation) period. We thus find
\begin{equation}
\label{FCbound}
\expect{\TP} \geq \frac{\lambda^2 \expect{B^2}}{2(1 - \rho)} + \rho.
\end{equation}
Combining~\eqref{FCbound} with~\eqref{glowb} gives the desired result.

Another way to explain~\eqref{FCbound} is that the average number of customers in a queue with vacations is at least the average number of customers in a queue without vacations, a standard M/G/1 queue.
\end{proof}
In order to investigate how tight the bounds derived in Theorem~\ref{boundg} are, we consider the case of exponentially distributed service times with mean~$1$. Further assume the vacation discipline of Example~\ref{vactime} with the vacation time distribution identical to the service time distribution. By Theorem~\ref{negativbin} we then find for $\DP(i)=1/(i+1)$ that $\expect{\TP}=2 \rho/(1-\rho)$, while Theorem~\ref{boundg} gives $\expect{\TP}\geq \rho/(1-\rho)$. So in this case the bound is off by a factor 2.
\begin{figure}[t]
  \centering
  \includegraphics[width=5in,keepaspectratio]{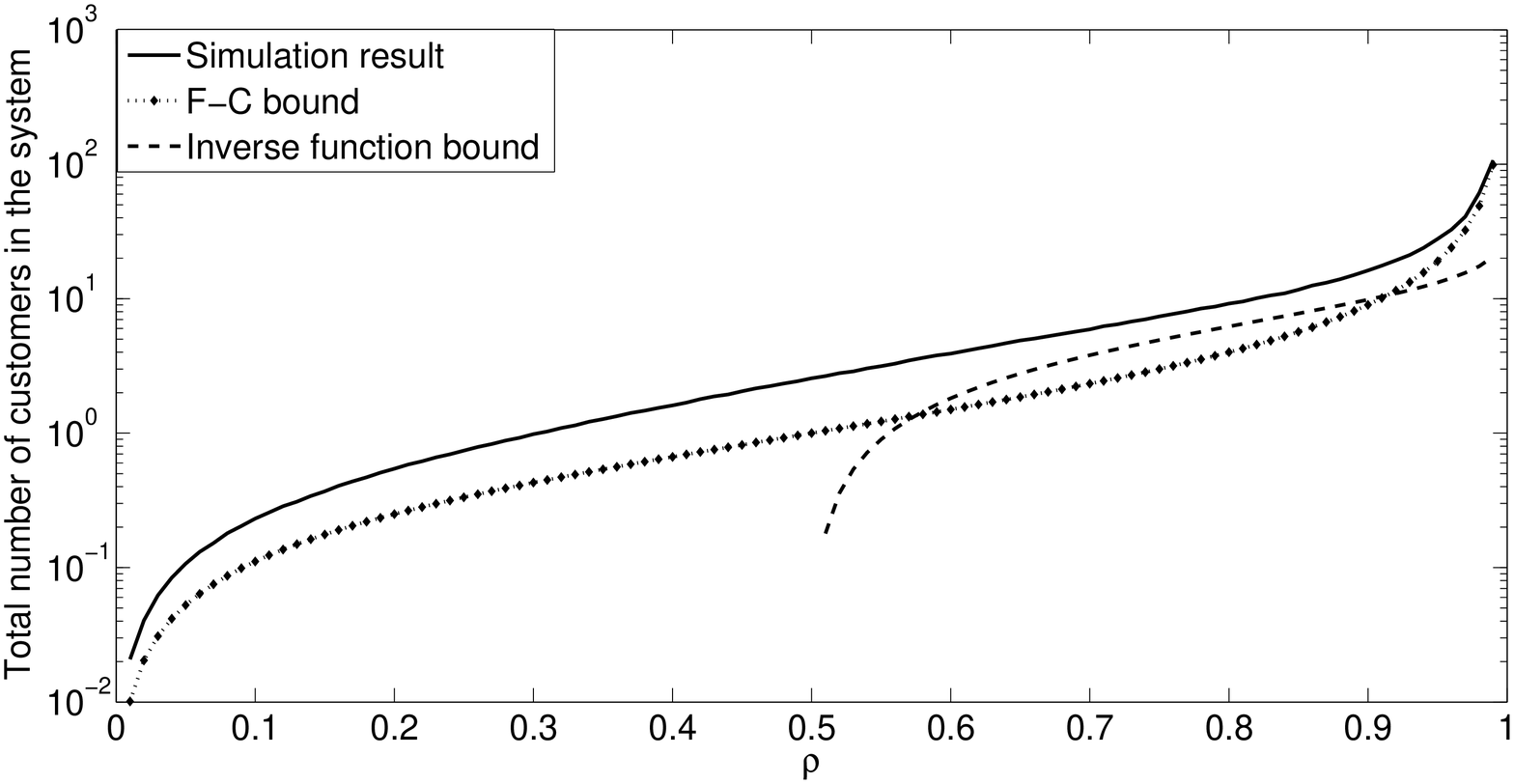}
  \caption{Average number of customers in the system for $\DP(i)=0.8^i$ and $\rho\in[0,1)$.}
  \label{a08deact}
\end{figure}
\begin{figure}[t]
  \centering
  \includegraphics[width=5in,keepaspectratio]{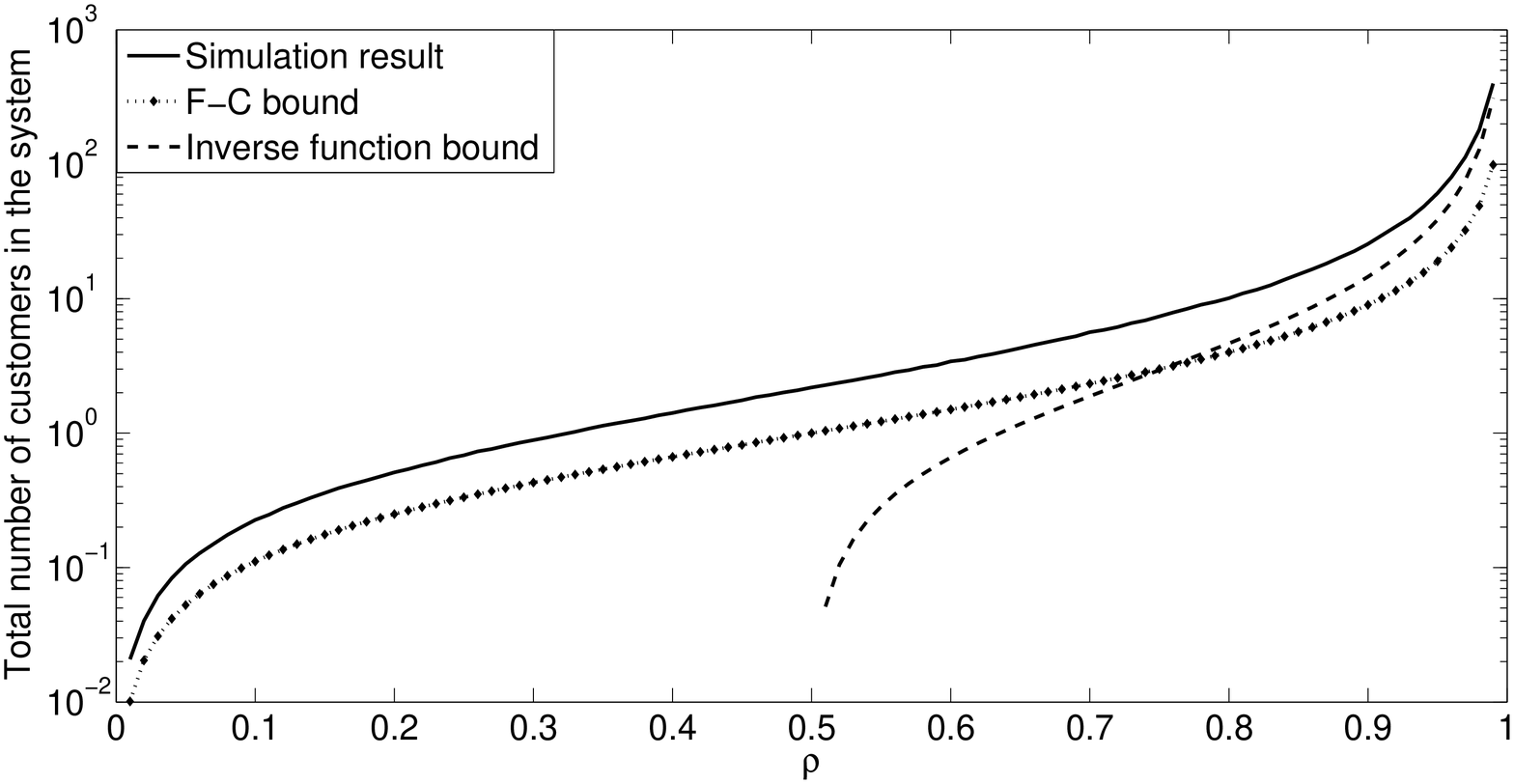}
  \caption{Average number of customers in the system for $\DP(i)=(1/(i+1))^{0.8}$ and $\rho\in[0,1)$.}
  \label{1overi08deact}
\end{figure}
We performed several numerical experiments for other de-activation probabilities $\DP(\cdot)$. Two typical results are given in Figures~\ref{a08deact} and~\ref{1overi08deact}, which show simulation results for the average number of customers in the system for $\DP(i)=0.8^i$ and $\DP(i)=(1/(i+1))^{0.8}$, respectively. We further added the two lower bounds derived in Theorem~\ref{boundg}, the inverse function bound~\eqref{glowb} and the bound that follows from the Fuhrmann-Cooper decomposition~\eqref{FCbound}. Note that we used a log-lin scale for graphical reasons.

For $\DP(i)=0.8^i$ we see that the simulation results are close to~\eqref{FCbound} for values of $\rho$ close to $1$, i.e.~the bound in Theorem~\ref{boundg} seems rather tight in heavy traffic and the average number of customers is close to the average number of customers in a standard M/G/1 queue.

For $\DP(i)=(1/(i+1))^{0.8}$ the simulation results are close to the inverse function bound~\eqref{glowb} for values of $\rho$ close to $1$, i.e.~the bound in Theorem~\ref{boundg} seems rather tight in heavy traffic for this choice for $\DP(\cdot)$ as well. In Section~\ref{HTresults} we will prove that the bound in Theorem~\ref{boundg} is asymptotically exact in heavy traffic for the cases considered here.

Next we compare the two processes $\{\TP(t),\sigma(t)\}_{t\geq 0}$ with de-activation probability $\DP(\cdot)$ and $\{\hat{\TP}(t),\hat{\sigma}(t)\}_{t\geq 0}$ with de-activation probability $\hat{\DP}(\cdot)$.

\begin{lemma}
\label{gcoup}
For the vacation discipline described in {\rm Example~\ref{vactime}}, and assuming that $ \hat{\DP}(i) \geq \DP(i)$, $i \geq 0$, $\hat{\TP}(0)=\TP(0)$ and $\hat{\sigma}(0)=\sigma(0)=0$, $\{\hat{\TP}(t)\}_{t\geq 0} \geq_{{\rm st}} \{\TP(t)\}_{t\geq 0}$.
\end{lemma}
\begin{proof}
The proof is based on a coupling $(\{\TP^{*}(t),\sigma^{*}(t)\}_{t\geq 0},\{\hat{\TP}^{*}(t),\hat{\sigma}^{*}(t)\}_{t\geq 0})$ between $\{\TP(t),\sigma(t)\}_{t\geq 0}$ and $\{\hat{\TP}(t),\hat{\sigma}(t)\}_{t\geq 0}$. That is, we construct the sample path of the coupled systems recursively such that, marginally, this sample path obeys the same probabilistic laws as the original process. Further we make sure that arrivals in both systems happen at the same time and that the $n$-th service takes the same amount of time in both systems. We also make sure the $n$-th `wait time', i.e.~the time $V$ described in Example~\ref{vactime}, is equal in both systems. Finally, we make sure that the system with de-activation probability $\hat{\DP}(\cdot)$ always de-activates if the system with de-activation probability $\DP(\cdot)$ de-activates, if the total amount of customers in both systems is equal and a service ends in both systems. This will make sure that $\hat{\TP}^{*}(t)\geq \TP^{*}(t)$ for all $t\geq 0$. A formal proof is given in Appendix~\ref{appendix}.
\end{proof}

Note that we only proved the result of Lemma~\ref{gcoup} for the vacation discipline of Example~\ref{vactime}. For general vacation disciplines, which may depend on the arrival process, Lemma~\ref{gcoup} does not always hold. It can for example be checked that Lemma~\ref{gcoup} does not hold for $G_{X_0}(r)=r^{100}$, $G_{X}(r)=r$, $\DP(i)=0$ for $i \geq 1$ and $\hat{\DP}(i)=0.1$ for $i \geq 1$.

Combining Lemma~\ref{gcoup} and Theorem~\ref{boundg} leads to a lower bound for the mean number of customers in a system with de-activation probability $\hat{\DP}(\cdot)$ if there exists a strictly decreasing convex function $\DP(\cdot)$ such that $\DP(i)\leq \hat{\DP}(i)$ for all $i$. We get
\[
\expect{\hat{L}} \geq \expect{L} \geq \max\Big\{\DP^{-1}\Big(\frac{1-\rho}{\expect{X}}\Big), \frac{\lambda^2\expect{B^2}}{2(1-\rho)}+\rho \Big\}.
\]

\subsection{State-dependent vacation lengths}
In this subsection we consider the vacation disciplines obeying Scenario~\ref{scen2}. For this class of vacation disciplines we find the following bounds.
\begin{theorem}
\label{boundf}
For {\rm Scenario~\ref{scen2}} and $\DP(i)=1$, $i \geq 0$,

${\rm (i)}$ If $f(\cdot)$ is a strictly increasing, unbounded and concave function,
\begin{equation}\label{boundfconc}
\expect{\TP} \geq \frac{\lambda^2 \expect{B^2}}{2(1 - \rho)} + \rho + f^{- 1}\Big(\frac{\lambda}{1 - \rho}\Big).
\end{equation}

${\rm (ii)}$ If $f(\cdot)$ is a strictly increasing convex function,
\begin{equation}\label{boundfconv}
\expect{\TP} \leq \frac{\lambda^2 \expect{B^2}}{2(1 - \rho)} + \rho + f^{- 1}\Big(\frac{\lambda}{1 - \rho}\Big).
\end{equation}
\end{theorem}
\begin{proof}
In steady state, the mean number of activations per unit of time equals the mean number of de-activations per unit of time, i.e.,
\begin{equation}
\label{actisdeactB}
\pr{\sigma=1}\frac{1}{\expect{B}}=\expect{f(\TP_I)}\pr{\sigma=0},
\end{equation}
where $\TP_I$ denotes the number of customers during a non-serving (vacation) period.

If $f(\cdot)$ is concave, it follows by Jensen's inequality that
\[
\expect{f(\TP_{I})} \leq f\left(\expect{\TP_{I}}\right).
\]
Since $f(\cdot)$ is increasing, we thus get, as $\pr{\sigma=1}=\rho$ and $\pr{\sigma=0}=1-\rho$,
\[
\expect{\TP_{I}} \geq f^{- 1}\left(\frac{\lambda}{1 - \rho}\right).
\]
The Fuhrmann-Cooper decomposition~\eqref{fuco1} implies
\[
\expect{\TP} = \frac{\lambda^2 \expect{B^2}}{2(1 - \rho)} + \rho + \expect{\TP_{I}},
\]
yielding~\eqref{boundfconc}.

The bound in equation~\eqref{boundfconv} follows by symmetry.
\end{proof}

Note that $f(i)=\nu i$ is both convex and concave. We thus find an exact result for this activation function,
\[
\expect{\TP} = \frac{2\lambda +\nu \lambda^2\expect{B^2}}{2\nu(1-\rho)}+\rho.
\]
This also follows from the generating function, which we derived in Theorem~\ref{nonhomlin}.

\begin{figure}[t]
  \centering
  \includegraphics[width=5in,keepaspectratio]{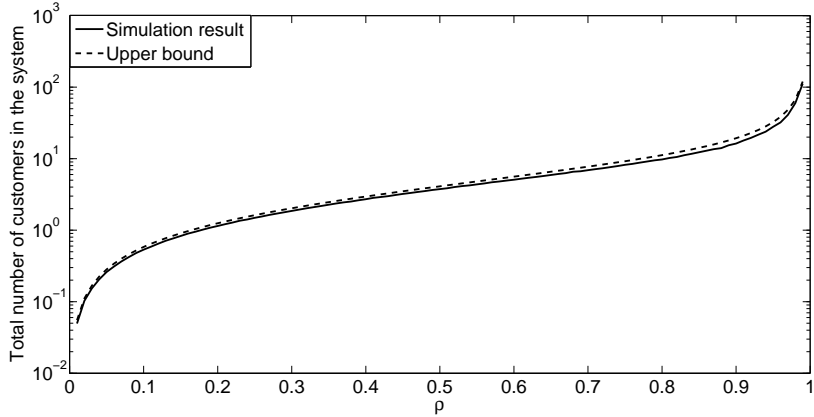}
  \caption{Average number of customers in the system for $f(i)=1.25^i-1$ and $\rho\in[0,1)$.}
  \label{b125act}
\end{figure}
\begin{figure}[t]
  \centering
  \includegraphics[width=5in,keepaspectratio]{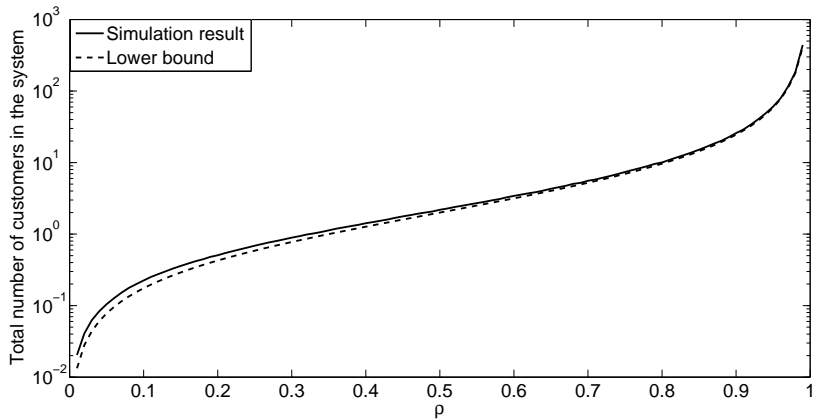}
  \caption{Average number of customers in the system for $f(i)=i^{0.8}$ and $\rho\in[0,1)$.}
  \label{fi08act}
\end{figure}
In order to investigate how tight the bounds in Theorem~\ref{boundf} are for activation functions other than $f(i) = \nu i$, we performed several numerical experiments. Two typical results are displayed in Figures~\ref{b125act} and~\ref{fi08act}, showing for exponentially distributed service times with mean one the average number of customers in the system for $f(i)=1.25^i-1$ and $f(i)=i^{0.8}$, respectively. For these activation functions we see that the simulated results are relatively close to their corresponding bounds for all values of $\rho$. In Section~\ref{HTresults} we will prove that the bounds in Theorem~\ref{boundf} are in fact asymptotically sharp in heavy traffic.

Next we compare the two processes $\{\TP(t),\sigma(t)\}_{t\geq 0}$ with activation rate $f(\cdot)$ and $\{\hat{\TP}(t),\hat{\sigma}(t)\}_{t\geq 0}$ with activation rate $\hat{f}(\cdot)$.

\begin{lemma}
\label{fcoup}
For {\rm Scenario~\ref{scen2}}, and assuming that $\hat{f}(i) \leq f(i)$, $\DP(i)=1$, $i\geq 0$, $\hat{\TP}(0)=\TP(0)$ and $\hat{\sigma}(0)=\sigma(0)=0$, $\{\hat{\TP}(t)\}_{t\geq 0} \geq_{{\rm st}} \{\TP(t)\}_{t\geq 0}$.
\end{lemma}
\begin{proof}
The proof of this lemma proceeds along similar lines as the proof of Lemma~\ref{gcoup}, i.e.~it is based on a coupling $(\{\TP^{*}(t),\sigma^{*}(t)\}_{t\geq 0},\{\hat{\TP}^{*}(t),\hat{\sigma}^{*}(t)\}_{t\geq 0})$ between $\{\TP(t),\sigma(t)\}_{t\geq 0}$ and $\{\hat{\TP}(t),\hat{\sigma}(t)\}_{t\geq 0}$. We construct the sample path of the coupled systems recursively such that, marginally, this sample path obeys the same probabilistic laws as the original process. Further we make sure that arrivals in both systems happen at the same time and that the $n$-th service takes the same amount of time in both systems. Finally we make sure that the system with activation rate $f(\cdot)$ always activates if the system with activation rate $\hat{f}(\cdot)$ activates, if the total amount of customers in both systems is equal and both systems are de-activated. This will make sure that $\hat{\TP}^{*}(t)\geq \TP^{*}(t)$ for all $t\geq 0$ as both systems always de-activate after one customer is served. A formal proof is given in Appendix~\ref{appendix}.
\end{proof}

Combining Lemma~\ref{fcoup} and Theorem~\ref{boundf} leads to an upper bound for the mean number of customers in a system with activation rate $\hat{f}(\cdot)$ if there exists a strictly increasing convex function $f(\cdot)$ such that $f(i)\leq \hat{f}(i)$ for all $i$. We get
\[
\expect{\hat{\TP}} \leq \expect{\TP} \leq \frac{\lambda^2 \expect{B^2}}{2(1 - \rho)} + \rho + f^{- 1}\Big(\frac{\lambda}{1 - \rho}\Big).
\]
Similarly we find a lower bound for the mean number of customers in a system if there exists a strictly increasing unbounded concave function $f(\cdot)$ such that $f(i)\geq \hat{f}(i)$ for all $i$.

\section{Heavy-traffic results}
\label{HTresults}
In this section we study the heavy-traffic behavior of the system. In particular, we derive the stationary distribution of the scaled number of customers in the system in heavy traffic, $L/\expect{L}$ for $\rho \uparrow 1$. More precisely, we let $\lambda$ vary and study the system when $\lambda$ approaches $1/\expect{B}$.

As an important byproduct, we obtain the limiting distribution of the stationary scaled sojourn time as $\rho \uparrow 1$ as well. For this we consider the Laplace-Stieltjes transform of $S/\expect{S}$, $\expect{\e^{-wS/\expect{S}}}$, with $w\geq 0$. By virtue of the distributional form of Little's law~\cite{KS88}, the PASTA property and a level crossings argument we know that
\[
G_{\TP}(r)=\expect{\e^{-\lambda(1-r)S}},
\]
or equivalently, for $\expect{L} \geq w$,
\[
\expect{\e^{-wS/\expect{S}}} = G_{\TP}(1-\frac{1}{\expect{L}}w) = \expect{(1-\frac{1}{\expect{L}}w)^{\expect{L} \frac{L}{\expect{L}}}}.
\]
Noting that $\expect{L} \rightarrow \infty$ as $\rho \uparrow 1$ and using a generalized version of the continuous mapping theorem, see e.g.~\cite{Kal97}, we then find as $\rho \uparrow 1$
\[
S/\expect{S} \cond W \mbox{ {\rm if and only if} }  L/\expect{L} \cond W.
\]
Here $W$ denotes some non-negative random variable and $\cond$ denotes convergence in distribution.

Note that $\TP$ and $X_i$ in general depend on the value of $\rho$, which is not fixed in this section. To emphasize this we will therefore write $G_{\TP}(r,\rho)$ for the generating function of $\TP$ and $G_{X_i}(r,\rho)$ for the generating function of $X_i$ in this section. Similarly we write $K(r,\rho)$ and $Y(r,\rho)$ for $K(\cdot)$ and $Y(\cdot)$ as defined in~\eqref{defK} and~\eqref{defY}. Further, in order to analyze the system in heavy traffic, we need to make some technical assumptions on the vacation distribution in heavy traffic. That is, in this section we will consider Scenario~\ref{scen3} as described below and Scenario~\ref{scen2} with functions $f(\cdot)$ that grow monotonically to infinity.
\begin{scenario}
\label{scen3}
$X_i\eqd X \stsl 0$ for all $i\geq 1$ and either $X_0 \eqd X$ and $\pr{X=0}=0$ or $X_0 \stsl X$ and $\pr{X=0}>0$. Further, $\expect{X_i^{*}} = \lim_{\rho \uparrow 1} \expect{X_i}$ exists and is finite for all $i\geq 0$, and
\begin{equation}
\label{HTassumpt}
\lim_{\rho \uparrow 1} \Big( \frac{\partial}{\partial \rho} G_{X_i}(r,\rho)\Big|_{r=\e^{-(1-\rho)u}} \Big)=0,
\end{equation}
for all $i\geq 0$ and $u\geq 0$.
\end{scenario}
The additional assumptions in Scenario~\ref{scen3} ensure that the vacation discipline behaves nicely when $\lambda$ approaches $1/\expect{B}$, i.e.~when $\rho \uparrow 1$ the vacation distribution for $\rho$ is similar to the vacation distribution for $\rho-\epsilon$ for small $\epsilon$, which is a desirable property from both a practical and theoretical perspective.

One example of a vacation discipline that belongs to Scenario~\ref{scen3} is the prototypical example of the back-off mechanism used in wireless networks described in Example~\ref{vactime}. For $i\geq 1$ we get
\[
\frac{\partial}{\partial \rho} G_{X_i}(r,\rho)= \frac{1-r}{\expect{B}}\tilde{V}'((1-r)\rho/\expect{B}),
\]
and thus~\eqref{HTassumpt} holds because $\expect{V}<\infty$. For $i=0$ it can be checked in a similar way that~\eqref{HTassumpt} holds.

Denote by ${\rm Exp}(\beta)$ a random variable having an exponential distribution with mean $1/\beta$, and denote by ${\rm \Gamma}(\alpha,\beta)$ a random variable having a gamma distribution with shape parameter $\alpha$ and rate parameter $\beta$. Define $R_B=\expect{B^2}/(2\expect{B})$, the mean residual service time, and $\upsilon_B = R_B/\expect{B}$.

\begin{theorem}
\label{eqaiHT}
For {\rm Scenario~\ref{scen3}} and $\DP(i)=a^i$ with $0\leq a <1$, $i\geq 0$,
\begin{equation}
\label{eqaiHTeq}
(1-\rho) L \cond {\rm Exp}(\upsilon_B^{-1}) \mbox{ {\rm as} } \rho\uparrow 1.
\end{equation}
\end{theorem}
\begin{proof}
Consider the Laplace-Stieltjes transform of $(1-\rho) \TP$, with $u\geq 0$, and note that
\begin{equation}
\label{LStoger}
\expect{\e^{-(1-\rho)u \TP}} = G_{\TP}(\e^{-(1-\rho)u},\rho).
\end{equation}
We will now use Theorem~\ref{eqai} to prove~\eqref{eqaiHTeq}. For this define $h(\rho)=\e^{-(1-\rho)u}$ and note that then,
\begin{equation}
\label{YroverY1}
\frac{Y(h(\rho),\rho)}{Y(1,\rho)} = \frac{(1-\rho) \tilde{B}(\rho(1-h(\rho))/\expect{B})(1-G_X(h(\rho),\rho))}{\expect{X}(\tilde{B}(\rho(1-h(\rho))/\expect{B}) - h(\rho)}.
\end{equation}
Applying l'H\^opital's rule twice,
\begin{equation}
\label{YroverYhop}
\lim_{\rho \uparrow 1} \frac{Y(h(\rho),\rho)}{Y(1,\rho)}=\frac{1}{1+\upsilon_B h'(1)}\Big(1 + \lim_{\rho \uparrow 1} \Big( \frac{1}{h'(\rho) \expect{X}}\frac{\partial}{\partial \rho} G_{X}(r,\rho)\Big|_{r=h(\rho)}\Big) \Big).
\end{equation}
By continuity of $Y(\cdot)$,
\[
\lim_{\rho \uparrow 1} Y(a^i h(\rho),\rho) = Y(a^i,1),
\]
for $i\geq 1$. From Theorem~\ref{eqai} we then find for $X_0 \eqd X$ that
\[
\lim_{\rho \uparrow 1} G_{\TP}(h(\rho),\rho)=\frac{1}{1+\upsilon_B h'(1)} = \frac{1}{1+ \upsilon_B u},
\]
which, by L\'evy's continuity theorem, gives~\eqref{eqaiHTeq} in case $X_0 \eqd X$.

We also find
\begin{align*}
&\lim_{\rho \uparrow 1} \Big(\frac{G_X(h(\rho),\rho)-G_{X_0}(h(\rho),\rho)}{1-G_X(h(\rho),\rho)}\Big) \\
&=\lim_{\rho \uparrow 1} \Big( \frac{\frac{\partial}{\partial \rho} G_{X}(r,\rho)+h'(\rho) \expect{X} - \frac{\partial}{\partial \rho} G_{X_0}(r,\rho) - h'(\rho) \expect{X_0}}{-\frac{\partial}{\partial \rho} G_{X}(r,\rho) -h'(\rho) \expect{X}} \Big|_{r= h(\rho)} \Big)\\
&=\frac{\expect{X_0}-\expect{X}}{\expect{X}} =\frac{K(1,\rho)}{Y(1,\rho)},
\end{align*}
and, for $i \geq 1$,
\[
\lim_{\rho \uparrow 1} \frac{G_X(a^i h(\rho),\rho)-G_{X_0}(a^i h(\rho),\rho)}{1-G_X(a^i h(\rho),\rho)} = \frac{G_X(a^i,1)-G_{X_0}(a^i,1)}{1-G_X(a^i,1)}.
\]
Further, from Theorem~\ref{eqai} we find, for $X_0 >_{{\rm st}} X$,
\[
G_{\TP}(h(\rho),\rho)=\frac{Y(h(\rho),\rho)}{Y(1,\rho)} \frac{ \sum_{j=0}^{\infty} \frac{G_{X}(a^j h(\rho),\rho ) - G_{X_0}(a^j h(\rho),\rho )}{1-G_{X}(a^j h(\rho),\rho)} \prod_{i=1}^j  Y(a^i h(\rho),\rho ) }{\sum_{j=0}^{\infty} \frac{G_{X}(a^j,\rho) - G_{X_0}(a^j,\rho)}{1-G_{X}(a^j,\rho)} \prod_{i=1}^j  Y(a^i,\rho)},
\]
as
\[
K(r,\rho)=Y(r,\rho) \frac{G_X(r,\rho)-G_{X_0}(r,\rho)}{1-G_X(r,\rho)}.
\]
We thus obtain
\[
\lim_{\rho \uparrow 1} G_{\TP}(h(\rho),\rho)=\frac{1}{1+ h'(1) \upsilon_B}= \frac{1}{1+ u \upsilon_B},
\]
which proves~\eqref{eqaiHTeq}.
\end{proof}
It is striking that the result in Theorem~\ref{eqaiHT} is independent of the precise assumption on when the server returns from a vacation. In fact, the behavior is similar to the heavy-traffic behavior of a standard M/G/1 queue without vacations~\cite{King62}.

Remember that in this paper we assume $\expect{B^2}<\infty$. This assumption is needed in the proof of Theorem~\ref{eqaiHT}, but not in the proof of Theorem~\ref{eqai}. If the service time distribution has a tail behavior like $t^{-k}$ with $1< k \leq 2$, i.e.~the service time has finite mean and infinite variance, we can prove along similar lines as the proof of Theorem~\ref{eqaiHT} that then the heavy-traffic behavior is similar to that of a standard M/G/1 queue without vacations as well~\cite{BC99}.

%
%

If the server de-activates less frequently than in Theorem~\ref{eqaiHT}, then one would expect the same result as in Theorem~\ref{eqaiHT}. The next theorem proves this result for the vacation discipline of Example~\ref{vactime}. Furthermore, we will prove a similar result for vacation disciplines in Scenario~\ref{scen2} with an aggressive activation function~$f(\cdot)$.
\begin{theorem}
For the vacation discipline described in {\rm Example~\ref{vactime}} and $\DP(i) \leq a^i$ with $a \in [0,1)$, $i\geq 0$, and for {\rm Scenario~\ref{scen2}} with $\DP(i)=1$, $i\geq 0$, and $f(\cdot)$ a strictly increasing continuous and convex function with $\lim_{i\rightarrow \infty} i^{-1}f^{-1}(i)=0$,
\begin{equation}
\label{LMG1HT}
(1-\rho) L \cond {\rm Exp}(\upsilon_B^{-1}) \mbox{ {\rm as} } \rho\uparrow 1.
\end{equation}
\end{theorem}
\begin{proof}
First assume the vacation discipline of Example~\ref{vactime} is used with $\DP(i) \leq a^i$, $a \in [0,1)$. By Lemma~\ref{gcoup} we have $\TP \leq_{{\rm st}} \TP_{a^i}$, where $\TP_{a^i}$ denotes a random variable with the steady-state distribution of the number of customers in the system with $\DP(i)=a^i$ for all $i$. Further, $\TP \geq_{{\rm st}} \TP_{{\rm M/G/1}}$. The result now follows from Theorem~\ref{eqaiHT}.

Now assume Scenario~\ref{scen2} with $\DP(i)=1$, $i\geq 0$, and $f(\cdot)$ a strictly increasing continuous and convex function with $\lim_{i\rightarrow \infty} i^{-1}f^{-1}(i)=0$. From~\eqref{boundfconv} we get, because $\lim_{i\rightarrow \infty} i^{-1}f^{-1}(i)=0$,
\begin{equation}
\label{LMG1HTeqp}
\lim_{\rho \uparrow 1} (1 - \rho) \expect{\TP} \leq \upsilon_B.
\end{equation}
Now consider the random variable $W=(1-\rho)(\TP - \TP_{{\rm M/G/1}})$ and note that $W$ is nonnegative because $\TP \geq_{{\rm st}} \TP_{{\rm M/G/1}}$ and $\rho < 1$. Therefore,
\[
\expect{|W|} = \expect{W} = \expect{(1-\rho) \TP} - \expect{(1-\rho) \TP_{{\rm M/G/1}}}.
\]
Thus as,
\[
\lim_{\rho \uparrow 1} (1 - \rho) \expect{\TP_{{\rm M/G/1}}} = \upsilon_B,
\]
we find from~\eqref{LMG1HTeqp} that $\expect{|W|}=0$, hence $W$ converges in mean to $0$. Using Slutsky's theorem we then get
\[
(1-\rho) \TP = W + (1-\rho) \TP_{{\rm M/G/1}} \cond {\rm Exp}(\upsilon_B^{-1}) \mbox{ {\rm as} } \rho\uparrow 1,
\]
which completes the proof.
\end{proof}

Next we consider vacation disciplines that are less aggressive. First we will consider Scenario~\ref{scen3} with $\DP(\cdot)$ inversely proportional to the queue length and Scenario~\ref{scen2} with a linear activation rate~$f(\cdot)$. For these vacation scenarios we will show that the heavy-traffic behavior does depend on the vacation scenario.

\begin{theorem}
\label{eq1overHT}
For {\rm Scenario~\ref{scen3}} and $\DP(i)=1/(i+1)$, $i\geq 0$,
\begin{equation}
\label{eq1overHTeq1}
(1-\rho) L \cond {\rm \Gamma}(1+\expect{X^{*}}\upsilon_B^{-1},\upsilon_B^{-1}) \mbox{ {\rm as} } \rho\uparrow 1.
\end{equation}
Similarly, for {\rm Scenario~\ref{scen2}} with $\DP(i)=1$ and $f(i)=\nu i$, $i\geq 0$,
\begin{equation}
\label{eq1overHTeq2}
(1-\rho) L \cond {\rm \Gamma}(1+1/(\nu R_B),\upsilon_B^{-1}) \mbox{ {\rm as} } \rho\uparrow 1.
\end{equation}
\end{theorem}
\begin{proof}
The proof for Scenario~\ref{scen3} proceeds along similar lines as the proof of Theorem~\ref{eqaiHT}. We consider the Laplace-Stieltjes transform of $(1-\rho) \TP$, $\expect{\e^{-(1-\rho)u \TP}}$, with $u\geq 0$, and use~\eqref{LStoger} and Theorem~\ref{eq1over} to prove~\eqref{eq1overHTeq1}. For this define $h(\rho)=\e^{-(1-\rho)u}$ and note that,
\begin{align*}
\lim_{\rho \uparrow 1} \int_{h(\rho)}^1 \frac{Y(x,\rho)}{x}dx &= \lim_{\rho \uparrow 1} \int_{h(\rho)}^1 \frac{\tilde{B}(\lambda(1-x))(1-G_{X}(x,\rho))}{x(\tilde{B}(\lambda(1-x))-x)}dx\\
&= \lim_{\rho \uparrow 1} \int_{0}^{1-h(\rho)} \frac{\tilde{B}(\lambda s)(1-G_{X}(1-s,\rho))}{(1-s)(\tilde{B}(\lambda s)-1+s)}ds.
\end{align*}
Using Taylor expansion and noting that $s = O(1-\rho)$ as $\rho \uparrow 1$ in the integration domain,
\begin{align*}
 \int_{h(\rho)}^1  \frac{Y(x,\rho)}{x}dx &= \frac{\expect{X^{*}}}{1-\rho} \int_0^{1-h(\rho)} \Big( \frac{1}{1+\frac{\rho^2 \upsilon_B s}{1-\rho}} + O(1-\rho) \Big) ds\\
&= \frac{\expect{X^{*}}}{1-\rho} \Big(\frac{1-\rho}{\rho^2 \upsilon_B} \log \Big(1 + \frac{\rho^2\upsilon_B(1-h(\rho))}{1-\rho}\Big) + O((1-\rho)^2) \Big).
\end{align*}
Thus,
\[
\lim_{\rho \uparrow 1} {\rm exp} \Big(-\int_{h(\rho)}^1  \frac{Y(x,\rho)}{x}dx\Big) = \Big(1 + u \upsilon_B \Big)^{-\expect{X^{*}}/\upsilon_B}.
\]
We now find for $X_0 \eqd X$, using $\eqref{YroverYhop}$ and Theorem~\ref{eq1over}, that
\begin{equation}
\label{HT1overXisX0}
\lim_{\rho \uparrow 1} G_{\TP}(h(\rho),\rho)=\Big(\frac{1}{1+  u \upsilon_B}\Big)^{1+\expect{X^{*}}/\upsilon_B},
\end{equation}
which, by L\'evy's continuity theorem, gives~\eqref{eq1overHTeq1} in case $X_0 \eqd X$.

For $X_0 \stsl X$, using~\eqref{defKalt} and~\eqref{eq1overeq2L0},
\begin{equation}
\label{pr02ndorder}
\pr{L=0} K(h(\rho),\rho) = \frac{G_{\TP_{{\rm M/G/1}}}(h(\rho),\rho) \Big( \expect{X_0} G_{X_0^{\rm {res}}}(h(\rho),\rho) - \expect{X} G_{X^{\rm {res}}}(h(\rho),\rho) \Big)}{\expect{X_0} - \expect{X} + \frac{\expect{X}}{1-\rho} \int_0^1 C(x,\rho) dx},
\end{equation}
with
\[
C(x,\rho)= G_{\TP_{{\rm M/G/1}}}(x) \Big( \expect{X_0} G_{X_0^{\rm {res}}}(x,\rho) - \expect{X} G_{X^{\rm {res}}}(x,\rho) \Big) \e^{\alpha(x,\rho)} .
\]
Thus $\lim_{\rho\uparrow 1} \pr{L=0} K(h(\rho),\rho) = 0$, as $C(x,\rho)>0$.

Similarly,
\[
\lim_{\rho \uparrow 1} \frac{1}{1-\rho} \frac{ \int_{0}^{h(\rho)} C(x,\rho) dx} {\expect{X_0} - \expect{X} + \frac{\expect{X}}{1-\rho} \int_0^1 C(x,\rho) dx} = \frac{1}{\expect{X}} .
\]
Hence by Theorem~\ref{eq1over} and~\eqref{HT1overXisX0} we find for $X_0 \stsl X$ that
\[
\lim_{\rho \uparrow 1} G_{\TP}(h(\rho),\rho)=\Big(\frac{1}{1+  u \upsilon_B}\Big)^{1+\upsilon_B\expect{X^{*}}}
\]
as well, which proves~\eqref{eq1overHTeq1}.

For Scenario~\ref{scen2} with $\DP(i)=1$ and $f(i)=\nu i$, $i\geq 0$, note that
\[
\int_{h(\rho)}^1 \frac{-\lambda(1-x)}{\nu (\tilde{B}(\lambda(1-x))-x)} dx = \frac{-\lambda}{\nu} \int_{0}^{1-h(\rho)} \frac{s}{\tilde{B}(\lambda s) - 1 + s}ds.
\]
Using Taylor expansion gives
\begin{align*}
&\int_{h(\rho)}^1 \frac{-\lambda(1-x)}{\nu (\tilde{B}(\lambda(1-x))-x)} dx = \frac{-\lambda}{\nu} \int_{0}^{1-h(\rho)} \Big( \frac{s}{-\rho s + \rho^2 \upsilon_B s^2+s} + O(1) \Big)  ds\\
&= \frac{-\rho}{\nu(1-\rho)\expect{B}} \int_{0}^{1-h(\rho)} \Big( \frac{1}{1+\frac{\rho^2 \upsilon_B s}{1-\rho}} + O(1-\rho) \Big)  ds\\
&= \frac{-\rho}{\nu(1-\rho)\expect{B}} \Big( \frac{1-\rho}{\rho^2 \upsilon_B} \log \Big( 1 + \frac{\rho^2 \upsilon_B (1-h(\rho))}{1-\rho} \Big) + O((1-\rho)^2) \Big).
\end{align*}
Thus,
\[
\lim_{\rho \uparrow 1} \int_{h(\rho)}^1 {\rm exp} \Big( \frac{-\lambda(1-x)}{\nu (\tilde{B}(\lambda(1-x))-x)} dx \Big) = \Big(1 + u \upsilon_B\Big)^{-1/(\nu R_B)},
\]
which, using Theorem~\ref{nonhomlin} and L\'evy's continuity theorem, gives~\eqref{eq1overHTeq2}.
\end{proof}
Instead of using the result in Theorem~\ref{eqaiHT} one could also use the differential equations~\eqref{difeq1over} and~\eqref{difeq1over2} to prove Theorem~\ref{eq1overHT} directly.

We thus see that the heavy-traffic behavior of $L$ does depend on the specific parameters if the vacation disciplines of Theorem~\ref{eq1overHT} are used. We further see that the number of customers still scales like $1-\rho$ in heavy traffic, i.e.~$(1-\rho) L$ converges to a random variable. This is not the case anymore for vacation disciplines that are even less aggressive as the next theorem states.
\begin{theorem}
\label{degenerate}
For the vacation discipline described in {\rm Example~\ref{vactime}} and $\DP(i) = 1/(i+1)^\alpha$ with $\alpha \in (0,1)$, $i\geq 0$, and for {\rm Scenario~\ref{scen2}} with $\DP(i)=1$ and $f(i)=\nu i^\alpha$ with $\alpha \in (0,1)$, $i\geq 0$,
\begin{equation}
\label{LoverELto1}
\frac{L}{\expect{L}} \cond 1 \mbox{ as } \rho\uparrow 1.
\end{equation}
In particular,
\[
\lim_{\rho \uparrow 1} (1-\rho)^{1/\alpha} \expect{L} = \expect{X^{*}}^{1/\alpha}
\]
for the vacation discipline described in {\rm Example~\ref{vactime}}, and
\[
\lim_{\rho \uparrow 1} (1-\rho)^{1/\alpha}  \expect{\TP} = \expect{B}^{-1/\alpha} \nu^{-1/\alpha}
\]
for {\rm Scenario~\ref{scen2}}.
\end{theorem}
We thus see for the vacation discipline described in Example~\ref{vactime} with the vacation probability inversely proportional to the queue length raised to the power $\alpha$, $\alpha \in (0,1)$, and for Scenario~\ref{scen2} with a linear activation rate raised to the power $\alpha$, $\alpha \in (0,1)$, that $(1-\rho) L$ diverges. In fact, we see that for these vacation disciplines the number of customers in the system scales like $(1-\rho)^{1/\alpha}$ in heavy traffic. We further see that, using this appropriate heavy-traffic scaling, the scaled number of users in the system in heavy traffic has a degenerate distribution.

In order to prove Theorem~\ref{degenerate} we first introduce some additional notation. For any function $g(\cdot):[0,\infty)\mapsto [0,\infty)$ define, for $a < 1$, $b > 1$ and $x \in [0, \infty)$,
\[
\gamma_{a, b}(x) =
\frac{(b - 1) g(a x) + (1 - a) g(b x)}{(b - a) g(x)}.
\]
Further define
\[
\kappa_{a, b} = 1 - \sup_x \gamma_{a, b}(x),
\]
and
\[
\chi_{a, b} = 1 - \inf_x \gamma_{a, b}(x).
\]
The proof of Theorem~\ref{degenerate} is based on the following proposition.
\begin{proposition}
\label{prop1}
Assume $g(\cdot)$ is concave and $\kappa_{a, b} >0$ for any $a < 1$ and $b > 1$, or $g(\cdot)$ is convex and $\chi_{a, b} < 0$ for any $a < 1$ and $b > 1$. If
\[
\lim_{\rho \uparrow 1} \frac{\expect{g(W)}}{g(\expect{W})} = 1
\]
then
\[
\frac{W}{\expect{W}} \cond 1 \mbox{ as } \rho \uparrow 1.
\]
\end{proposition}
The proof of Proposition~\ref{prop1} is deferred to Appendix~\ref{appendix} as it relies on a few technical lemmas that are relegated from the main text.

Having established Proposition~\ref{prop1}, we can now prove Theorem~\ref{degenerate}.
\begin{proof}
(of Theorem~\ref{degenerate})
For the vacation discipline described in Example~\ref{vactime} and $\DP(i) = 1/(i+1)^\alpha$ with $\alpha \in (0,1)$, $i\geq 0$, we know from Theorem~\ref{boundg} that
\[
\expect{\TP} \geq \DP^{-1}\Big(\frac{1-\rho}{\expect{X}}\Big),
\]
or, as $\DP^{-1}(i)=i^{-1/\alpha} - 1$,
\begin{equation}
\label{lboundEL}
\lim_{\rho \uparrow 1} (1-\rho)^{1/\alpha} \expect{X}^{-1/\alpha} \expect{\TP} \geq 1.
\end{equation}

Now consider the system with $\hat{\DP}(i)=1$ for $i \leq \lceil \DP^{-1}(\beta) \rceil$ and $\hat{\DP}(i)=\beta$ for $i > \lceil \DP^{-1}(\beta) \rceil$, where $\beta>0$ and, for stability, $\beta<(1-\rho)/\expect{X}$. Thus, by construction, $\hat{\DP}(i) \geq \DP(i)$ for all $i$. Further assume that this system uses a vacation discipline similar to that of  Example~\ref{vactime}, but with a slight modification; the server only activates if at least $\lceil \DP^{-1}(\beta) \rceil + 1$ customers are present in the system, instead of at least $1$. That is, we have $\lceil \DP^{-1}(\beta) \rceil$ permanent customers. It follows immediately from Lemma~\ref{gcoup} that $\hat{\TP} \geq_{{\rm st}} \TP$.
Further, using~\eqref{geng} we find
\begin{equation}
\label{gerfuncfixedalpha}
G_{\hat{\TP}}(r) = \frac{\pr{\hat{L}=\lceil \DP^{-1}(\beta) \rceil} \tilde{B}(\lambda(1-r))(G_{X}(r)-G_{X_0}(r))}{\tilde{B}(\lambda(1-r))-r-(1-G_{X}(r))\beta} r^{\lceil \DP^{-1}(\beta) \rceil},
\end{equation}
with
\[
\pr{\hat{\TP}=\lceil \DP^{-1}(\beta) \rceil}=\frac{(1-\rho-\beta \lambda \expect{V})(1-\tilde{V}(\lambda))}{\lambda \expect{V} \tilde{V}(\lambda)}.
\]
Now take $\beta=\frac{(1-\rho)(1-\delta)}{\expect{X}}$, $\delta>0$, and note that from~\eqref{gerfuncfixedalpha}
\[
\expect{\hat{\TP}}=\lceil \DP^{-1}(\frac{(1-\rho)(1-\delta)}{\expect{X}}) \rceil + C(\rho),
\]
with $\lim_{\rho \uparrow 1} C(\rho) (1-\rho) < \infty$. Using that $\DP^{-1}(i)=i^{-1/\alpha}-1$, we get
\[
(1-\rho)^{1/\alpha} \expect{X}^{-1/\alpha}\expect{\hat{\TP}} \leq  (1-\delta)^{-1/\alpha} + (1-\rho)^{1/\alpha} \expect{X}^{-1/\alpha} C(\rho),
\]
and, hence, as $0 < \alpha < 1$,
\[
\lim_{\rho \uparrow 1} (1-\rho)^{1/\alpha} \expect{X}^{-1/\alpha} \expect{\hat{\TP}} \leq \lim_{\rho \uparrow 1} (1-\delta)^{-1/\alpha} + (1-\rho)^{1/\alpha} \expect{X}^{-1/\alpha} C(\rho) \leq (1-\delta)^{-1/\alpha},
\]
for any $\delta >0$. Thus, as $\hat{\TP} \geq_{{\rm st}} \TP$, we find
\[
\lim_{\rho \uparrow 1} (1-\rho)^{1/\alpha} \expect{X}^{-1/\alpha} \expect{\TP} \leq 1.
\]
Therefore, using equation~\eqref{lboundEL},
\begin{equation}
\label{meaneq1}
\lim_{\rho \uparrow 1} (1-\rho)^{1/\alpha} \expect{X}^{-1/\alpha} \expect{\TP} = 1,
\end{equation}
or
\begin{equation}
\label{meaneq12}
\lim_{\rho \uparrow 1} \frac{\DP(\expect{\TP})}{1-\rho} = \frac{1}{\expect{X}}.
\end{equation}
From~\eqref{actisdeactA} we find
\[
\expect{\DP(\TP)}=\frac{1}{\expect{X}} (1-\rho-\pr{\TP=0}),
\]
and hence, because $\pr{\TP=0}/(1-\rho) \to 0$ as $\rho \uparrow 1$ by using Lemma~\ref{gcoup} and an argument similar to~\eqref{pr02ndorder},
\[
\lim_{\rho \uparrow 1} \frac{\expect{\DP(\TP)}}{1-\rho} = \frac{1}{\expect{X}}.
\]
Combining this with~\eqref{meaneq12} we find
\[
\lim_{\rho \uparrow 1} \frac{\expect{\DP(\TP)}}{\DP(\expect{\TP})} = 1.
\]
Further, because $\DP(\cdot)$ is strictly convex,
\begin{align*}
\gamma_{a, b}(x) &=\frac{(b - 1) (1+ax)^{-\alpha} + (1 - a) (1+bx)^{-\alpha}}{(b - a) (1+x)^{-\alpha}} \\
&=\frac{b-1}{b-a} \Big(\frac{1+ax}{1+x}\Big)^{-\alpha} + \frac{1-a}{b-a}\Big(\frac{1+bx}{1+x}\Big)^{-\alpha} \\
&> \Big(\frac{b-1}{b-a} \frac{1+ax}{1+x} + \frac{1-a}{b-a} \frac{1+bx}{1+x}\Big)^{-\alpha}=1,
\end{align*}
and the statement for the vacation discipline described in Example~\ref{vactime} follows from Proposition~\ref{prop1}.

The proof for Scenario~\ref{scen2} with $\DP(i)=1$ and $f(\cdot)=\nu i^\alpha$ with $\alpha \in (0,1)$, $i\geq 0$, proceeds along similar lines. First, using Theorem~\ref{boundf} we know
\[
\expect{\TP} \geq \frac{\lambda^2 \expect{B^2}}{2(1 - \rho)} + \rho + f^{- 1}\Big(\frac{\lambda}{1 - \rho}\Big),
\]
or, as $f^{-1}(i)=(i/\nu)^{1/\alpha}$,
\begin{equation}
\label{lboundELf}
\lim_{\rho \uparrow 1} (1-\rho)^{1/\alpha} \expect{B}^{1/\alpha} \nu^{1/\alpha} \expect{\TP} \geq 1.
\end{equation}

Now consider the system with $\hat{f}(i)=0$ for $i \leq \lceil f^{-1}(\beta) \rceil$ and $\hat{f}(i)=\beta$ for $i > \lceil f^{-1}(\beta) \rceil$. where $\beta>0$ and, for stability, $\beta>\frac{\rho}{(1-\rho)\expect{B}}$. Thus, by construction, $\hat{f}(i) \leq f(i)$ for all $i\geq 0$. Further take $\hat{\DP}(i)=1$ for all $i \geq 0$. We know from Lemma~\ref{fcoup} that $\hat{\TP} \geq_{{\rm st}} \TP$.

We can find the generating function of $\hat{\TP}$ using~\eqref{geng}. We can also find this generating function by noting that this system behaves as an M/G/1 queue with $\lceil f^{-1}(g) \rceil$ permanent customers and service requirement $B+{\rm Exp}(\beta)$, i.e.~the time required to serve a customer is the sum of the vacation time and the service time.

Now take $\beta=\frac{\rho}{(1-\rho)(1-\delta)\expect{B}}$, $\delta>0$, which gives
\[
\expect{\hat{\TP}}=\lceil f^{-1}(\frac{\rho}{(1-\rho)(1-\delta)\expect{B}}) \rceil + C(\rho),
\]
with $\lim_{\rho \uparrow 1} C(\rho) (1-\rho) < \infty$. We then find in a similar way as before that
\[
\lim_{\rho \uparrow 1} (1-\rho)^{1/\alpha} \expect{B}^{1/\alpha} \nu^{1/\alpha} \expect{\TP} \leq 1,
\]
and hence, using~\eqref{lboundELf},
\[
\lim_{\rho \uparrow 1} (1-\rho)^{1/\alpha} \expect{B}^{1/\alpha} \nu^{1/\alpha} \expect{\TP} = 1.
\]
Noting that $\TP_{{\rm M/G/1}}/\expect{\TP} \cond 0$ as $\rho \uparrow 1$ we thus find, using the Fuhrmann-Cooper decomposition~\eqref{fuco1},
\[
\lim_{\rho \uparrow 1} (1-\rho) f(\expect{\TP_I}) = \frac{1}{\expect{B}}.
\]
Further, from~\eqref{actisdeactB} we find
\[
\lim_{\rho \uparrow 1} (1-\rho) \expect{f(\TP_I)} = \frac{1}{\expect{B}},
\]
so that
\begin{equation}
\label{finaleqdeg}
\lim_{\rho \uparrow 1} \frac{\expect{f(\TP_I)}}{\DP(\expect{\TP_I})} = 1.
\end{equation}
Finally, because $f(\cdot)$ is strictly concave,
\[
\gamma_{a, b}(x) =\frac{(b - 1) a^\alpha + (1 - a) b^{\alpha}}{b - a} < \Big(\frac{b-1}{b-a} a + \frac{1-a}{b-a} b\Big)^{\alpha}=1,
\]
and we find~\eqref{LoverELto1} by invoking~\eqref{finaleqdeg} and Proposition~\ref{prop1}.
\end{proof}
The proof of Theorem~\ref{degenerate} can be simplified if we assume $\expect{B^3}<\infty$ and $\expect{X_i^3}<\infty$ for all $i\geq 0$. In that case we can find $\expect{\TP^2}$ along similar lines as we found $\expect{\TP}$ in the proof of Theorem~\ref{degenerate}. It then follows that $\lim_{\rho\uparrow 1} \expect{\TP^2}/\expect{\TP}^2 = 1$ in this case, so that the assertion in Theorem~\ref{degenerate} follows from Chebyshev's inequality.
\section{Conclusions}
\label{concl}
In this paper we have obtained results for queues with random back-offs. Such random back-offs can be modeled by rates of activating during a back-off period $f(\cdot)$ and by the probability of initiating a back-off period after a service completion $\DP(\cdot)$. For various choices of $f(\cdot)$ and $\DP(\cdot)$, and under some additional assumptions, we have obtained exact expressions for the distribution of the number of customers in the system in Section~\ref{exactana} and bounds for the mean stationary number of customers in the system in Section~\ref{secbounds}. These results were employed to derive heavy-traffic limit theorems in Section~\ref{HTresults}, which showed the existence of a clear trichotomy, that can be best explained through the function $\DP(\cdot)$. Clearly, in order for the system to be stable when $\rho\uparrow 1$, $\DP(\cdot)$ should eventually, as the number of customers increases, go to zero. This condition is also sufficient, see Lemma~\ref{stable}. Roughly speaking (for details and further assumptions see Section~\ref{HTresults}), the queueing system with back-offs can display three modes of operation, depending on the asymptotic decay rate of $\DP(\cdot)$. These three modes can be understood as follows:

(i) The case $\DP(i)=1/(i+1)$ represents the {\it balanced regime}, in which the heavy-traffic behavior is influenced by both the system behavior without back-offs, and the back-off periods. Hence, large queue sizes typically build up according to sample paths that display exceptional (interrupted) busy periods and exceptional sequences of back-off periods. The number of customers $L$ is of the order $O((1-\rho)^{-1})$, and the more detailed information in Theorem~\ref{eq1overHT} reveals a gamma distribution containing information of the arrival process, service times, and the back-off function.

(ii) When $\DP(\cdot)$ decays faster than $1/(i+1)$, for instance $\DP(i)=O(a^{i})$ with $a\in(0,1)$, it is shown that the heavy-traffic behavior of the system is as if the back-off periods do not exist. The intuition is that when $\rho\uparrow 1$, the system spends most of the time in states of large queue sizes in which the probability of initiating a back-off becomes negligible. Indeed, in Theorem~\ref{eqaiHT} it is shown that for $\DP(i)=a^{i}$ with $a\in(0,1)$ the heavy-traffic behavior of the system is the same as that of an M/G/1 system without back-offs.

(iii)  When $\DP(\cdot)$ decays slower than $1/(i+1)$, for instance for $\DP(i)=O((i+1)^{-a})$ with $a\in(0,1)$, it is shown that the heavy-traffic behavior of the system is completely determined by the back-offs. Theorem~\ref{degenerate} says that for $\DP(i)=1/(i+1)^a$ with $a\in(0,1)$ the mean number of customers is $O((1-\rho)^{-1/a})$, while the stationary distribution of the number of customers is degenerate and thus strongly concentrated around its mean. Hence, for systems with such back-offs, the heavy-traffic behavior is entirely different from that of systems without back-off (the M/G/1 system in this case).

Another relevant observation that follows from the analysis is that the order of the number of customers in the system~$\TP$ in heavy traffic is independent of the mean length of the vacation and service times in all three modes.

The revealed trichotomy for the single-node system provides some important insights for the wireless networks equipped with back-off rules similar as discussed in Section~\ref{intro}, because the single-node system provides a {\it best-case scenario} for networks with multiple nodes. To see this, first notice that for a network to be in heavy traffic, the aggregated traffic intensity in some clique, a set of nodes of which at most one can be active at the same time, tends to $1$. The total number of packets in this clique behaves like the number of packets in the corresponding single-node system with two modifications.
First, the probability to go into back-off is based on a subset of all customers. Hence the network will be in back-off more often if $\DP(\cdot)$ is decreasing and the total number of customers in both systems were equal.
Second, the length of the vacation period is the minimum over the back-off lengths of all non-blocked nodes, which might change during the vacation period. We thus see that the vacation length is at least equal to the minimum back-off length of a node in the clique assuming none of the nodes is prevented from activating. Hence, taking this minimum as the actual vacation length in the corresponding single-node system, we see that vacations in the network always take at least as long as in the single-node system.
As both modifications intuitively have a negative impact on the delay performance, it seems reasonable to assume that the total number of packets in the network is at best equal to the total number of customers in the corresponding single-node system.

We thus see that more aggressive activation schemes can potentially improve the delay performance. On the other hand, however, these aggressive activation schemes may fail to achieve maximum stability and hence are unstable in heavy traffic. Maximum stability for general networks is only guaranteed when $\DP(i)=O(1/(\log(i)+1))$ (see \cite{GS10,JSSW10,RSS09,SST11}) which, based on the analysis in this paper, might result in very poor delay performance in heavy traffic. An interesting topic for further research is to establish for which scenarios the delay performance in the network is roughly equal to the delay performance of the corresponding single-node system.
\section{Acknowledgments}
This work was supported by Microsoft Research through its PhD Scholarship Programme, an ERC starting grant and a TOP grant from NWO.

We thank J.A.C. Resing for bringing the work of Sevast'yanov to our attention.


\bibliographystyle{plain}

%
%
%
%
%

\appendix
\section{Preliminary results and proofs}
\label{appendix}

This appendix contains a few technical lemmas and some proofs that have been relegated from the main text. To make this appendix self-contained we restate some results from the main text.

\begin{lemma}
\label{convproofs}
${\rm (i)}$ If $X_0 \eqd X$, then,
\[
\prod_{i=0}^{\infty}Y(a^ir),
\]
with $0 \leq a <1$, converges for all $r\in [0,1]$.

${\rm (ii)}$ If $X_0 >_{{\rm st}} X$, then,
\[
\sum_{j=0}^{\infty}K(a^jr)\prod_{i=0}^{j-1}Y(a^ir),
\]
with $0 \leq a <1$, converges for all $r\in [0,1]$.
\end{lemma}
\begin{proof}
To prove case~${\rm (i)}$ first note that this infinite product converges if and only if
\[
\sum_{i=0}^{\infty}(Y(a^ir)-1)
\]
converges. To prove convergence of this infinite series we will use the ratio test (d'Alembert's criterion). We have, with $h(r)=\tilde{B}(\lambda(1-r))$ and $k(r)=\tilde{B}(\lambda(1-r))G_X(r)$,
\[
\lim_{i\rightarrow \infty} \left|\frac{Y(a^{i+1}r)-1}{Y(a^ir)-1}\right| = \lim_{i\rightarrow \infty} \frac{(-a^i r + h(a^i r ))(a^{i+1} r - k(a^{i+1}r ))}{(-a^{i+1} r + h(a^{i+1} r ))(a^{i} r - k(a^{i}r))} = \lim_{i\rightarrow \infty} \frac{a^{i+1} r - k(a^{i+1}r)}{a^{i} r - k(a^{i}r)}.
\]
By l'H\^opital's rule,
\[
\lim_{i\rightarrow \infty} \frac{a^{i+1} r - k(a^{i+1}r)}{a^{i} r - k(a^{i}r)} =
\lim_{i\rightarrow \infty} a \frac{1+\lambda G_X(a^{i+1}r)\tilde{B}'(\lambda(1-a^{i+1}r))-\tilde{B}(\lambda(1-a^{i+1}r))G'_X(a^{i+1}r)}
{1+\lambda G_X(a^{i}r)\tilde{B}'(\lambda(1-a^{i}r))-\tilde{B}(\lambda(1-a^{i}r))G'_X(a^{i+1}r)}.
\]

We thus find
\[
\lim_{i\rightarrow \infty} \left|\frac{Y(a^{i+1}r)-1}{Y(a^ir)-1}\right| = a < 1,
\]
proving case~${\rm (i)}$.

For case~${\rm (ii)}$ note that
\[
\lim_{n\rightarrow \infty} \frac{K(a^{n+1}r) \prod_{i=0}^{n}Y(a^ir)}{K(a^{n}r) \prod_{i=0}^{n-1}Y(a^ir)}=Y(0)<1,
\]
for all $r$ as $0\leq a < 1$. Thus, by the ratio test, the series in case~${\rm (ii)}$ converges.
\end{proof}


\begin{replemma}{gcoup}
For the vacation discipline described in {\rm Example~\ref{vactime}}, and assuming that $ \hat{\DP}(i) \geq \DP(i)$, $i \geq 0$, $\hat{\TP}(0)=\TP(0)$ and $\hat{\sigma}(0)=\sigma(0)=0$, $\{\hat{\TP}(t)\}_{t\geq 0} \geq_{{\rm st}} \{\TP(t)\}_{t\geq 0}$.
\end{replemma}
\begin{proof}
To prove this lemma we will construct a coupling $(\{\TP^{*}(t),\sigma^{*}(t)\}_{t\geq 0},\{\hat{\TP}^{*}(t),\hat{\sigma}^{*}(t)\}_{t\geq 0})$ between $\{\TP(t),\sigma(t)\}_{t\geq 0}$ and $\{\hat{\TP}(t),\hat{\sigma}(t)\}_{t\geq 0}$ such that $\hat{\TP}^{*}(t)\geq \TP^{*}(t)$ for all $t\geq 0$. The result then follows.

Let $A$, $B$ and $V$ be (infinite) vectors of realizations of independent random variables, where $A_i$ is exponentially distributed with parameter $\lambda$, $B_i$ is generally distributed with distribution function $F_B(\cdot)$ and $V_i$ is generally distributed with distribution function $F_V(\cdot)$. Further let $N_V(t)$ and $N_B(t)$ be the total number of activations and service completions of the process belonging to $\DP(\cdot)$. Similarly, let $\hat{N}_V(t)$ and $\hat{N}_B(t)$ be the total number of activations and service completions of the process belonging to $\hat{\DP}(\cdot)$. We will construct a coupling such that $\hat{\TP}^{*}(t)=\TP^{*}(t)$ and $\hat{\sigma}^{*}(t)\leq\sigma^{*}(t)$ or $\hat{\TP}^{*}(t)>\TP^{*}(t)$, $\hat{N}_V^{*}(t)\geq N_V^{*}(t)$ and $\hat{N}_B^{*}(t)\leq N_B^{*}(t)$, for all $t\geq 0$.

Denote by $R(t)$ the remaining time until an activation or service completion in the process belonging to $\DP(\cdot)$ at time $t$ and, similarly, denote by $\hat{R}(t)$ the remaining time until an activation or service completion in the process belonging to $\hat{\DP}(\cdot)$. We make arrivals occur simultaneously in both processes and denote by $J(t)$ the remaining time until an arrival. Initially set $R(\tau_0)=\hat{R}(\tau_0)=V_1$ and $J(\tau_0)=A_1$.

Define the jump epochs $0\equiv \tau_0 < \tau_1 < \dots$. The jump epochs and the coupling are constructed recursively and we inductively prove the statement in the construction of this coupling. First take $\TP^{*}(\tau_0)=\TP(\tau_0)$, $\sigma^{*}(\tau_0)=\sigma(\tau_0)$, $\hat{\TP}^{*}(\tau_0)=\hat{\TP}(\tau_0)$ and $\hat{\sigma}^{*}(\tau_0)=\hat{\sigma}(\tau_0)$ and note that $\hat{\TP}^{*}(\tau_0) = \TP^{*}(\tau_0)$ and $\hat{\sigma}^{*}(\tau_0)=\sigma^{*}(\tau_0)$ by assumption. Also, $\hat{N}_V^{*}(\tau_0)= N_V^{*}(\tau_0)=\hat{N}_B^{*}(\tau_0)=N_B^{*}(\tau_0)=0$.

Now assume $\hat{\TP}^{*}(\tau_i)= \TP^{*}(\tau_i)$ and $\hat{\sigma}^{*}(\tau_i)\leq\sigma^{*}(\tau_i)$ or $\hat{\TP}^{*}(\tau_i)>\TP^{*}(\tau_i)$, $\hat{N}_V^{*}(\tau_i)\geq N_V^{*}(\tau_i)$ and $\hat{N}_B^{*}(\tau_i)\leq N_B^{*}(\tau_i)$ for some $i \in \mathbb{N}_0$. Set $\tau_{i+1}=\tau_i+\min\{R(\tau_i),\hat{R}(\tau_i),J(\tau_i)\}$ and $\TP^{*}(t)=\TP^{*}(\tau_i)$, $\hat{\TP}^{*}(t)=\hat{\TP}^{*}(\tau_i)$, $\sigma^{*}(t)=\sigma^{*}(\tau_i)$, $\hat{\sigma}^{*}(t)=\hat{\sigma}^{*}(\tau_i)$, $N_V^{*}(t)=N_V^{*}(\tau_i)$, $\hat{N}_V^{*}(t)=\hat{N}_V^{*}(\tau_i)$, $N_B^{*}(t)=N_B^{*}(\tau_i)$ and $\hat{N}_B^{*}(t)=\hat{N}_B^{*}(\tau_i)$ for all $t\in(\tau_i,\tau_{i+1})$. So, by the induction hypothesis, $\hat{\TP}^{*}(t)= \TP^{*}(t)$ and $\hat{\sigma}^{*}(t)\leq\sigma^{*}(t)$ or $\hat{\TP}^{*}(t)>\TP^{*}(t)$, $\hat{N}_V^{*}(t)\geq N_V^{*}(t)$ and $\hat{N}_B^{*}(t)\leq N_B^{*}(t)$ for all $\tau_i\leq t < \tau_{i+1}$. To define the values at time $\tau_{i+1}$ we distinguish nine cases.

Case 1: $J(\tau_i)=\min\{R(\tau_i),\hat{R}(\tau_i),J(\tau_i)\}$. Set $\TP^{*}(\tau_{i+1})=\TP^{*}(\tau_{i})+1$, $\hat{\TP}^{*}(\tau_{i+1})=\hat{\TP}^{*}(\tau_{i})+1$, $\sigma^{*}(\tau_{i+1})=\sigma^{*}(\tau_i)$, $\hat{\sigma}^{*}(\tau_{i+1})=\hat{\sigma}^{*}(\tau_i)$, $N_V^{*}(\tau_{i+1})=N_V^{*}(\tau_i)$, $\hat{N}_V^{*}(\tau_{i+1})=\hat{N}_V^{*}(\tau_i)$, $N_B^{*}(\tau_{i+1})=N_B^{*}(\tau_i)$ and $\hat{N}_B^{*}(\tau_{i+1})=\hat{N}_B^{*}(\tau_i)$. Further set $R(\tau_{i+1})=R(\tau_i)-\tau_{i+1}+\tau_i$, $\hat{R}(\tau_{i+1})=\hat{R}(\tau_i)-\tau_{i+1}+\tau_i$ and $J(\tau_{i+1})=A_{i+1}$.

Case 2: $R(\tau_i)=\min\{R(\tau_i),\hat{R}(\tau_i),J(\tau_i)\}$, $R(\tau_i)=\hat{R}(\tau_i)$ and $\sigma^{*}(\tau_i)=\hat{\sigma}^{*}(\tau_i)=0$. Set $\TP^{*}(\tau_{i+1})=\TP^{*}(\tau_{i})$, $\hat{\TP}^{*}(\tau_{i+1})=\hat{\TP}^{*}(\tau_{i})$, $\sigma^{*}(\tau_{i+1})={\rm I}_{\{\TP^{*}(\tau_{i})>0\}}$, $\hat{\sigma}^{*}(\tau_{i+1})={\rm I}_{\{\hat{\TP}^{*}(\tau_{i})>0\}}$, $N_V^{*}(\tau_{i+1})=N_V^{*}(\tau_i)+1$, $\hat{N}_V^{*}(\tau_{i+1})=\hat{N}_V^{*}(\tau_i)+1$, $N_B^{*}(\tau_{i+1})=N_B^{*}(\tau_i)$ and $\hat{N}_B^{*}(\tau_{i+1})=\hat{N}_B^{*}(\tau_i)$. Further set $R(\tau_{i+1})=B_{N_{B}^{*}(\tau_{i+1})+1} \sigma^{*}(\tau_{i+1}) + V_{N_{V}^{*}(\tau_{i+1})+1} (1-\sigma^{*}(\tau_{i+1}))$, $\hat{R}(\tau_{i+1})=B_{\hat{N}_{B}^{*}(\tau_{i+1})+1} \hat{\sigma}^{*}(\tau_{i+1}) + V_{\hat{N}_{V}^{*}(\tau_{i+1})+1}(1-\hat{\sigma}^{*}(\tau_{i+1}))$ and $J(\tau_{i+1})=J(\tau_{i})-\tau_{i+1}+\tau_i$.

Case 3: $R(\tau_i)=\min\{R(\tau_i),\hat{R}(\tau_i),J(\tau_i)\}$, $R(\tau_i)=\hat{R}(\tau_i)$ and $\sigma^{*}(\tau_i)=\hat{\sigma}^{*}(\tau_i)=1$. Let $U_i$ be a realization of a random variable that is uniformly distributed in $[0,1]$. Now set $\TP^{*}(\tau_{i+1})=\TP^{*}(\tau_{i})-1$, $\hat{\TP}^{*}(\tau_{i+1})=\hat{\TP}^{*}(\tau_{i})-1$, $\sigma^{*}(\tau_{i+1})={\rm I}_{\{\DP(\TP^{*}(\tau_{i+1}))<U_i\}}$, $\hat{\sigma}^{*}(\tau_{i+1})={\rm I}_{\{\hat{\DP}(\hat{\TP}^{*}(\tau_{i+1}))<U_i\}}$, $N_V^{*}(\tau_{i+1})=N_V^{*}(\tau_i)$, $\hat{N}_V^{*}(\tau_{i+1})=\hat{N}_V^{*}(\tau_i)$, $N_B^{*}(\tau_{i+1})=N_B^{*}(\tau_i)+1$ and $\hat{N}_B^{*}(\tau_{i+1})=\hat{N}_B^{*}(\tau_i)+1$. Further set $R(\tau_{i+1})=B_{N_{B}^{*}(\tau_{i+1})+1} \sigma^{*}(\tau_{i+1}) + V_{N_{V}^{*}(\tau_{i+1})+1} (1-\sigma^{*}(\tau_{i+1}))$, $\hat{R}(\tau_{i+1})=B_{\hat{N}_{B}^{*}(\tau_{i+1})+1} \hat{\sigma}^{*}(\tau_{i+1}) + V_{\hat{N}_{V}^{*}(\tau_{i+1})+1}(1-\hat{\sigma}^{*}(\tau_{i+1}))$ and $J(\tau_{i+1})=J(\tau_{i})-\tau_{i+1}+\tau_i$.

Case 4: $R(\tau_i)=\min\{R(\tau_i),\hat{R}(\tau_i),J(\tau_i)\}$, $R(\tau_i)<\hat{R}(\tau_i)$ and $\sigma^{*}(\tau_i)=0$. Set $\TP^{*}(\tau_{i+1})=\TP^{*}(\tau_{i})$, $\hat{\TP}^{*}(\tau_{i+1})=\hat{\TP}^{*}(\tau_{i})$, $\sigma^{*}(\tau_{i+1})={\rm I}_{\{\TP^{*}(\tau_{i})>0\}}$, $\hat{\sigma}^{*}(\tau_{i+1})=\hat{\sigma}^{*}(\tau_{i+1})$, $N_V^{*}(\tau_{i+1})=N_V^{*}(\tau_i)+1$, $\hat{N}_V^{*}(\tau_{i+1})=\hat{N}_V^{*}(\tau_i)$, $N_B^{*}(\tau_{i+1})=N_B^{*}(\tau_i)$ and $\hat{N}_B^{*}(\tau_{i+1})=\hat{N}_B^{*}(\tau_i)$. Further set $R(\tau_{i+1})=B_{N_{B}^{*}(\tau_{i+1})+1} \sigma^{*}(\tau_{i+1}) + V_{N_{V}^{*}(\tau_{i+1})+1} (1-\sigma^{*}(\tau_{i+1}))$, $\hat{R}(\tau_{i+1})=\hat{R}(\tau_i)-\tau_{i+1}+\tau_i$ and $J(\tau_{i+1})=J(\tau_{i})-\tau_{i+1}+\tau_i$.

Case 5: $R(\tau_i)=\min\{R(\tau_i),\hat{R}(\tau_i),J(\tau_i)\}$, $R(\tau_i)<\hat{R}(\tau_i)$ and $\sigma^{*}(\tau_i)=1$. Let $U_i$ be a realization of a random variable that is uniformly distributed in $[0,1]$. Now set $\TP^{*}(\tau_{i+1})=\TP^{*}(\tau_{i})-1$, $\hat{\TP}^{*}(\tau_{i+1})=\hat{\TP}^{*}(\tau_{i})$, $\sigma^{*}(\tau_{i+1})={\rm I}_{\{\DP(\TP^{*}(\tau_{i+1}))<U_i\}}$, $\hat{\sigma}^{*}(\tau_{i+1})=\hat{\sigma}^{*}(\tau_{i+1})$, $N_V^{*}(\tau_{i+1})=N_V^{*}(\tau_i)$, $\hat{N}_V^{*}(\tau_{i+1})=\hat{N}_V^{*}(\tau_i)$, $N_B^{*}(\tau_{i+1})=N_B^{*}(\tau_i)+1$ and $\hat{N}_B^{*}(\tau_{i+1})=\hat{N}_B^{*}(\tau_i)$. Further set $R(\tau_{i+1})=B_{N_{B}^{*}(\tau_{i+1})+1} \sigma^{*}(\tau_{i+1}) + V_{N_{V}^{*}(\tau_{i+1})+1} (1-\sigma^{*}(\tau_{i+1}))$, $\hat{R}(\tau_{i+1})=\hat{R}(\tau_i)-\tau_{i+1}+\tau_i$ and $J(\tau_{i+1})=J(\tau_{i})-\tau_{i+1}+\tau_i$.

Case 6: $\hat{R}(\tau_i)=\min\{R(\tau_i),\hat{R}(\tau_i),J(\tau_i)\}$, $R(\tau_i)>\hat{R}(\tau_i)$ and $\hat{\sigma}^{*}(\tau_i)=0$. Set $\TP^{*}(\tau_{i+1})=\TP^{*}(\tau_{i})$, $\hat{\TP}^{*}(\tau_{i+1})=\hat{\TP}^{*}(\tau_{i})$, $\sigma^{*}(\tau_{i+1})=\sigma^{*}(\tau_{i})$, $\hat{\sigma}^{*}(\tau_{i+1})={\rm I}_{\{\hat{\TP}^{*}(\tau_{i})>0\}}$, $N_V^{*}(\tau_{i+1})=N_V^{*}(\tau_i)$, $\hat{N}_V^{*}(\tau_{i+1})=\hat{N}_V^{*}(\tau_i)+1$, $N_B^{*}(\tau_{i+1})=N_B^{*}(\tau_i)$ and $\hat{N}_B^{*}(\tau_{i+1})=\hat{N}_B^{*}(\tau_i)$. Further set $R(\tau_{i+1})=R(\tau_i)-\tau_{i+1}+\tau_i$, $\hat{R}(\tau_{i+1})=B_{\hat{N}_{B}^{*}(\tau_{i+1})+1} \hat{\sigma}^{*}(\tau_{i+1}) + V_{\hat{N}_{V}^{*}(\tau_{i+1})+1} (1-\hat{\sigma}^{*}(\tau_{i+1}))$ and $J(\tau_{i+1})=J(\tau_{i})-\tau_{i+1}+\tau_i$.

Case 7: $\hat{R}(\tau_i)=\min\{R(\tau_i),\hat{R}(\tau_i),J(\tau_i)\}$, $R(\tau_i)>\hat{R}(\tau_i)$ and $\hat{\sigma}^{*}(\tau_i)=1$. Let $U_i$ be a realization of a random variable that is uniformly distributed in $[0,1]$. Now set $\TP^{*}(\tau_{i+1})=\TP^{*}(\tau_{i})$, $\hat{\TP}^{*}(\tau_{i+1})=\hat{\TP}^{*}(\tau_{i})-1$, $\sigma^{*}(\tau_{i+1})=\sigma^{*}(\tau_{i+1})$, $\hat{\sigma}^{*}(\tau_{i+1})={\rm I}_{\{\hat{\DP}(\hat{\TP}^{*}(\tau_{i+1}))<U_i\}}$, $N_V^{*}(\tau_{i+1})=N_V^{*}(\tau_i)$, $\hat{N}_V^{*}(\tau_{i+1})=\hat{N}_V^{*}(\tau_i)$, $N_B^{*}(\tau_{i+1})=N_B^{*}(\tau_i)$ and $\hat{N}_B^{*}(\tau_{i+1})=\hat{N}_B^{*}(\tau_i)+1$. Further set $R(\tau_{i+1})=R(\tau_i)-\tau_{i+1}+\tau_i$, $\hat{R}(\tau_{i+1})=B_{\hat{N}_{B}^{*}(\tau_{i+1})+1} \hat{\sigma}^{*}(\tau_{i+1}) + V_{\hat{N}_{V}^{*}(\tau_{i+1})+1} (1-\hat{\sigma}^{*}(\tau_{i+1}))$ and $J(\tau_{i+1})=J(\tau_{i})-\tau_{i+1}+\tau_i$.

Case 8: $R(\tau_i)=\min\{R(\tau_i),\hat{R}(\tau_i),J(\tau_i)\}$, $R(\tau_i)=\hat{R}(\tau_i)$, $\sigma^{*}(\tau_i)=0$ and $\hat{\sigma}^{*}(\tau_i)=1$. Let $U_i$ be a realization of a random variable that is uniformly distributed in $[0,1]$. Now set $\TP^{*}(\tau_{i+1})=\TP^{*}(\tau_{i})$, $\hat{\TP}^{*}(\tau_{i+1})=\hat{\TP}^{*}(\tau_{i})-1$, $\sigma^{*}(\tau_{i+1})={\rm I}_{\{\TP^{*}(\tau_{i})>0\}}$, $\hat{\sigma}^{*}(\tau_{i+1})={\rm I}_{\{\hat{\DP}(\hat{\TP}^{*}(\tau_{i}))<U_i\}}$, $N_V^{*}(\tau_{i+1})=N_V^{*}(\tau_i)+1$, $\hat{N}_V^{*}(\tau_{i+1})=\hat{N}_V^{*}(\tau_i)$, $N_B^{*}(\tau_{i+1})=N_B^{*}(\tau_i)$ and $\hat{N}_B^{*}(\tau_{i+1})=\hat{N}_B^{*}(\tau_i)+1$. Further set $R(\tau_{i+1})=B_{N_{B}^{*}(\tau_{i+1})+1} \sigma^{*}(\tau_{i+1}) + V_{N_{V}^{*}(\tau_{i+1})+1} (1-\sigma^{*}(\tau_{i+1}))$, $\hat{R}(\tau_{i+1})=B_{\hat{N}_{B}^{*}(\tau_{i+1})+1} \hat{\sigma}^{*}(\tau_{i+1}) + V_{\hat{N}_{V}^{*}(\tau_{i+1})+1}(1-\hat{\sigma}^{*}(\tau_{i+1}))$ and $J(\tau_{i+1})=J(\tau_{i})-\tau_{i+1}+\tau_i$.

Case 9: $R(\tau_i)=\min\{R(\tau_i),\hat{R}(\tau_i),J(\tau_i)\}$, $R(\tau_i)=\hat{R}(\tau_i)$, $\sigma^{*}(\tau_i)=1$ and $\hat{\sigma}^{*}(\tau_i)=0$. Let $U_i$ be a realization of a random variable that is uniformly distributed in $[0,1]$. Now set $\TP^{*}(\tau_{i+1})=\TP^{*}(\tau_{i})-1$, $\hat{\TP}^{*}(\tau_{i+1})=\hat{\TP}^{*}(\tau_{i})$, $\sigma^{*}(\tau_{i+1})={\rm I}_{\{\DP(\TP^{*}(\tau_{i}))<U_i\}}$, $\hat{\sigma}^{*}(\tau_{i+1})={\rm I}_{\{\hat{\TP}^{*}(\tau_{i})>0\}}$, $N_V^{*}(\tau_{i+1})=N_V^{*}(\tau_i)$, $\hat{N}_V^{*}(\tau_{i+1})=\hat{N}_V^{*}(\tau_i)+1$, $N_B^{*}(\tau_{i+1})=N_B^{*}(\tau_i)+1$ and $\hat{N}_B^{*}(\tau_{i+1})=\hat{N}_B^{*}(\tau_i)$. Further set $R(\tau_{i+1})=B_{N_{B}^{*}(\tau_{i+1})+1} \sigma^{*}(\tau_{i+1}) + V_{N_{V}^{*}(\tau_{i+1})+1} (1-\sigma^{*}(\tau_{i+1}))$, $\hat{R}(\tau_{i+1})=B_{\hat{N}_{B}^{*}(\tau_{i+1})+1} \hat{\sigma}^{*}(\tau_{i+1}) + V_{\hat{N}_{V}^{*}(\tau_{i+1})+1}(1-\hat{\sigma}^{*}(\tau_{i+1}))$ and $J(\tau_{i+1})=J(\tau_{i})-\tau_{i+1}+\tau_i$.

Note that the remaining cases occur with probability zero as the random variables in the vector $A$ are exponentially distributed.

From the sample path construction we can deduce that
\begin{equation}
\label{difTP}
\hat{\TP}^{*}(\tau_i)-{\TP^{*}}(\tau_i)=N_B^{*}(\tau_i)-\hat{N}_B^{*}(\tau_i).
\end{equation}
Thus $\hat{\TP}^{*}(\tau_i)={\TP^{*}}(\tau_i)$ if and only if $N_B^{*}(\tau_i)=\hat{N}_B^{*}(\tau_i)$. Further we can deduce that
\begin{eqnarray}
\label{difR}
\hat{R}(\tau_i)-{R}(\tau_i)&=&\sum_{j=N_V^{*}(\tau_i)+1}^{\hat{N}_V^{*}(\tau_i)} V_j - \sum_{k=\hat{N}_B^{*}(\tau_i)+1}^{N_B^{*}(\tau_i)} B_k +(1-\hat{\sigma}^{*}(\tau_i))V_{\hat{N}_V^{*}(\tau_i)+1} \\
&& + \hat{\sigma}^{*}(\tau_i)B_{\hat{N}_B^{*}(\tau_i)+1} -(1-\sigma^{*}(\tau_i))V_{N_V^{*}(\tau_i)+1} - \sigma^{*}(\tau_i)B_{N_B^{*}(\tau_i)+1}.
\end{eqnarray}

We now need to prove that $\hat{\TP}^{*}(\tau_{i+1})= \TP^{*}(\tau_{i+1})$ and $\hat{\sigma}^{*}(\tau_{i+1})\leq\sigma^{*}(\tau_{i+1})$ or $\hat{\TP}^{*}(\tau_{i+1})>\TP^{*}(\tau_{i+1})$, $\hat{N}_V^{*}(\tau_{i+1})\geq N_V^{*}(\tau_{i+1})$ and $\hat{N}_B^{*}(\tau_{i+1})\leq N_B^{*}(\tau_{i+1})$ in all nine cases.

For case 1 this follows immediately from the induction hypothesis.

For case 2 note that $\hat{\sigma}^{*}(\tau_{i+1})>\sigma^{*}(\tau_{i+1})$ only if $\hat{\TP}^{*}(\tau_{i+1})>\TP^{*}(\tau_{i+1})=0$, so that the statement holds in this case as well.

For case 3 note that $\hat{\DP}(\hat{\TP}^{*}(\tau_{i}))\geq \DP(\TP^{*}(\tau_{i}))$, so that $\hat{\sigma}^{*}(\tau_{i+1})\leq\sigma^{*}(\tau_{i+1})$ if $\hat{\TP}^{*}(\tau_{i+1})= \TP^{*}(\tau_{i+1})$. The statement now follows.

For case 4 first assume that $\hat{\TP}^{*}(\tau_i)={\TP^{*}}(\tau_i)$. Then it follows from the induction hypothesis that $\sigma^{*}(\tau_i)=\hat{\sigma}^{*}(\tau_i)=0$, and from equation \eqref{difTP} it follows that $N_B^{*}(\tau_i)=\hat{N}_B^{*}(\tau_i)$. Then, for $R(\tau_i)<\hat{R}(\tau_i)$ to hold we need $\hat{N}_V^{*}(\tau_{i})> N_V^{*}(\tau_{i})$ as follows from equation \eqref{difR}. If $\hat{\TP}^{*}(\tau_i)>{\TP^{*}}(\tau_i)$ we have $N_B^{*}(\tau_i)>\hat{N}_B^{*}(\tau_i)$. Thus we again need $\hat{N}_V^{*}(\tau_{i})> N_V^{*}(\tau_{i})$ in order to have $R(\tau_i)<\hat{R}(\tau_i)$. The statement now follows.

For case 5 the statement follows immediately from the induction hypothesis.

For case 6 first assume that $\hat{\TP}^{*}(\tau_i)={\TP^{*}}(\tau_i)$ and $\sigma^{*}(\tau_i)=\hat{\sigma}^{*}(\tau_i)=0$. Following the same reasoning as for case 4 this yields that we need $\hat{N}_V^{*}(\tau_{i}) < N_V^{*}(\tau_{i})$ to have $R(\tau_i)>\hat{R}(\tau_i)$, contradicting the induction hypothesis. Thus it not possible to be in case 6 if $\hat{\TP}^{*}(\tau_i)={\TP^{*}}(\tau_i)$ and $\sigma^{*}(\tau_i)=\hat{\sigma}^{*}(\tau_i)=0$. For all other situations the statement is easily seen to hold.

For case 7 first assume that $\hat{\TP}^{*}(\tau_i)={\TP^{*}}(\tau_i)$ and $\sigma^{*}(\tau_i)=\hat{\sigma}^{*}(\tau_i)=1$. Following the same reasoning as in case 6 this would yield that we need $\hat{N}_V^{*}(\tau_{i}) < N_V^{*}(\tau_{i})$, which contradicts the induction hypothesis. Thus it is only possible to be in case 7 if $\hat{\TP}^{*}(\tau_{i})>\TP^{*}(\tau_{i})$, and hence $N_B^{*}(\tau_i)>\hat{N}_B^{*}(\tau_i)$.

For case 8 note that $\hat{\TP}^{*}(\tau_i)>{\TP^{*}}(\tau_i)$, as $\sigma^{*}(\tau_i)<\hat{\sigma}^{*}(\tau_i)$. Thus, using equation \eqref{difTP}, $N_B^{*}(\tau_i)>\hat{N}_B^{*}(\tau_i)$ and, using equation \eqref{difR}, $N_V^{*}(\tau_i)<\hat{N}_V^{*}(\tau_i)$. The statement now follows as $\hat{\sigma}^{*}(\tau_{i+1})>\sigma^{*}(\tau_{i+1})$ only if $\hat{\TP}^{*}(\tau_{i+1})>\TP^{*}(\tau_{i+1})=0$.

For case 9 the statement follows immediately from the induction hypothesis.

Finally, it can be verified that the marginal distributions of $\{\TP^{*}(t),\sigma^{*}(t)\}_{t\geq 0}$ and $\{\hat{\TP}^{*}(t),\hat{\sigma}^{*}(t)\}_{t\geq 0}$ are the same as the distribution of $\{\TP(t),\sigma(t)\}_{t\geq 0}$ and $\{\hat{\TP}(t),\hat{\sigma}(t)\}_{t\geq 0}$.
\end{proof}

\begin{replemma}{fcoup}
For {\rm Scenario~\ref{scen2}}, and assuming that $\hat{f}(i) \leq f(i)$, $\DP(i)=1$, $i\geq 0$, $\hat{\TP}(0)=\TP(0)$ and $\hat{\sigma}(0)=\sigma(0)=0$, $\{\hat{\TP}(t)\}_{t\geq 0} \geq_{{\rm st}} \{\TP(t)\}_{t\geq 0}$.
\end{replemma}
\begin{proof}
The proof of this lemma proceeds along similar lines as the proof of Lemma~\ref{gcoup}. That is, we will construct a coupling $(\{\TP^{*}(t),\sigma^{*}(t)\}_{t\geq 0},\{\hat{\TP}^{*}(t),\hat{\sigma}^{*}(t)\}_{t\geq 0})$ between $\{\TP(t),\sigma(t)\}_{t\geq 0}$ and $\{\hat{\TP}(t),\hat{\sigma}(t)\}_{t\geq 0}$ such that $\hat{\TP}^{*}(t)\geq \TP^{*}(t)$ for all $t\geq 0$. The result then follows.

We will construct the sample path of the coupled systems recursively such that, marginally, this sample path obeys the same probabilistic laws as the original process. Further we make sure that arrivals in both systems happen at the same time and that the $n$-th service takes the same amount of time in both systems. Finally we make sure that the system with activation rate $f(\cdot)$ always activates if the system with activation rate $\hat{f}(\cdot)$ activates, if the total number of customers in both systems is equal and both systems are de-activated. This will ensure that $\hat{\TP}^{*}(t)\geq \TP^{*}(t)$ for all $t\geq 0$ as both systems always de-activate after one customer is served. Now we will formally construct this coupling.

Let $A$, $B$ and $V$ be (infinite) vectors of realizations of independent random variables, where $A_i$ is exponentially distributed with parameter $\lambda$, $B_i$ is generally distributed with distribution function $F_B(\cdot)$ and $V_i$ is exponentially distributed with parameter $1$. Further let $N_B(t)$ be the total number of service completions of the process belonging to $f(\cdot)$. Similarly, let $\hat{N}_B(t)$ be the total number of service completions of the process belonging to $\hat{f}(\cdot)$.

Denote by $R(t)$ the remaining time until an activation or service completion in the process belonging to $f(\cdot)$ at time $t$ and, similarly, denote by $\hat{R}(t)$ the remaining time until an activation or service completion in the process belonging to $\hat{f}(\cdot)$. We make arrivals occur simultaneously in both processes and denote by $J(t)$ the remaining time until an arrival.

We will construct a coupling such that $\hat{\TP}^{*}(t)>\TP^{*}(t)$, or $\hat{\TP}^{*}(t)=\TP^{*}(t)$ and $0=\hat{\sigma}^{*}(t)<\sigma^{*}(t)=1$, or $\hat{\TP}^{*}(t)=\TP^{*}(t)$, $\hat{\sigma}^{*}(t)=\sigma^{*}(t)$ and $\hat{R}(t)\geq R(t)$, for all $t\geq 0$.

Define the jump epochs $0\equiv \tau_0 < \tau_1 < \dots$. The jump epochs and the coupling are constructed recursively and we inductively prove the statement in the construction of this coupling. First take $\TP^{*}(\tau_0)=\TP(\tau_0)$, $\sigma^{*}(\tau_0)=\sigma(\tau_0)$, $\hat{\TP}^{*}(\tau_0)=\hat{\TP}(\tau_0)$ and $\hat{\sigma}^{*}(\tau_0)=\hat{\sigma}(\tau_0)$ and note that $\hat{\TP}^{*}(\tau_0) = \TP^{*}(\tau_0)$ and $\hat{\sigma}^{*}(\tau_0)=\sigma^{*}(\tau_0)$ by assumption. Further set $\hat{N}_B^{*}(\tau_0)=N_B^{*}(\tau_0)=0$, $J(\tau_0)=A_1$ and $R(\tau_0)=\hat{R}(\tau_0)=V_1/f(\TP^{*}(\tau_0))$, where $1/0\equiv\infty$.

Now assume $\hat{\TP}^{*}(\tau_i)>\TP^{*}(\tau_i)$, or $\hat{\TP}^{*}(\tau_i)=\TP^{*}(\tau_i)$ and $0=\hat{\sigma}^{*}(\tau_i)<\sigma^{*}(\tau_i)=1$, or $\hat{\TP}^{*}(\tau_i)=\TP^{*}(\tau_i)$, $\hat{\sigma}^{*}(\tau_i)=\sigma^{*}(\tau_i)$ and $\hat{R}(\tau_i)\geq R(\tau_i)$, for some $i \in \mathbb{N}_0$. Set $\tau_{i+1}=\tau_i+\min\{R(\tau_i),\hat{R}(\tau_i),J(\tau_i)\}$ and $\TP^{*}(t)=\TP^{*}(\tau_i)$, $\hat{\TP}^{*}(t)=\hat{\TP}^{*}(\tau_i)$, $\sigma^{*}(t)=\sigma^{*}(\tau_i)$, $\hat{\sigma}^{*}(t)=\hat{\sigma}^{*}(\tau_i)$, $R(t)=R(\tau_i)-t+\tau_i$ and $\hat{R}(t)=\hat{R}(\tau_i)-t+\tau_i$ for all $t\in(\tau_i,\tau_{i+1})$. So, by the induction hypothesis, $\hat{\TP}^{*}(t)>\TP^{*}(t)$, or $\hat{\TP}^{*}(t)=\TP^{*}(t)$ and $0=\hat{\sigma}^{*}(t)<\sigma^{*}(t)=1$, or $\hat{\TP}^{*}(t)=\TP^{*}(t)$, $\hat{\sigma}^{*}(t)=\sigma^{*}(t)$ and $\hat{R}(t)\geq R(t)$ for all $\tau_i\leq t < \tau_{i+1}$. To define the values at time $\tau_{i+1}$ we distinguish seven cases.

Case 1: $J(\tau_i)=\min\{R(\tau_i),\hat{R}(\tau_i),J(\tau_i)\}$. Set $\TP^{*}(\tau_{i+1})=\TP^{*}(\tau_{i})+1$, $\hat{\TP}^{*}(\tau_{i+1})=\hat{\TP}^{*}(\tau_{i})+1$, $\sigma^{*}(\tau_{i+1})=\sigma^{*}(\tau_i)$, $\hat{\sigma}^{*}(\tau_{i+1})=\hat{\sigma}^{*}(\tau_i)$, $N_B^{*}(\tau_{i+1})=N_B^{*}(\tau_i)$ and $\hat{N}_B^{*}(\tau_{i+1})=\hat{N}_B^{*}(\tau_i)$. Further set $R(\tau_{i+1})=(R(\tau_i)-\tau_{i+1}+\tau_i)\sigma^{*}(\tau_{i+1})+(1-\sigma^{*}(\tau_{i+1}))V_{i+1}/f(\TP^{*}(\tau_{i+1}))$, $\hat{R}(\tau_{i+1})=(\hat{R}(\tau_i)-\tau_{i+1}+\tau_i)\hat{\sigma}^{*}(\tau_{i+1})+(1-\hat{\sigma}^{*}(\tau_{i+1}))V_{i+1}/\hat{f}(\hat{\TP}^{*}(\tau_{i+1}))$ and $J(\tau_{i+1})=A_{i+1}$.

Case 2: $R(\tau_i)=\min\{R(\tau_i),\hat{R}(\tau_i),J(\tau_i)\}$, $R(\tau_i)=\hat{R}(\tau_i)$ and $\sigma^{*}(\tau_i)=\hat{\sigma}^{*}(\tau_i)=0$. Set $\TP^{*}(\tau_{i+1})=\TP^{*}(\tau_{i})$, $\hat{\TP}^{*}(\tau_{i+1})=\hat{\TP}^{*}(\tau_{i})$, $\sigma^{*}(\tau_{i+1})=1$, $\hat{\sigma}^{*}(\tau_{i+1})=1$, $N_B^{*}(\tau_{i+1})=N_B^{*}(\tau_i)$ and $\hat{N}_B^{*}(\tau_{i+1})=\hat{N}_B^{*}(\tau_i)$. Further set $R(\tau_{i+1})=B_{N_{B}^{*}(\tau_{i+1})+1}$, $\hat{R}(\tau_{i+1})=B_{\hat{N}_{B}^{*}(\tau_{i+1})+1}$ and $J(\tau_{i+1})=J(\tau_{i})-\tau_{i+1}+\tau_i$.

Case 3: $R(\tau_i)=\min\{R(\tau_i),\hat{R}(\tau_i),J(\tau_i)\}$, $R(\tau_i)=\hat{R}(\tau_i)$ and $\sigma^{*}(\tau_i)=\hat{\sigma}^{*}(\tau_i)=1$. Set $\TP^{*}(\tau_{i+1})=\TP^{*}(\tau_{i})-1$, $\hat{\TP}^{*}(\tau_{i+1})=\hat{\TP}^{*}(\tau_{i})-1$, $\sigma^{*}(\tau_{i+1})=0$, $\hat{\sigma}^{*}(\tau_{i+1})=0$, $N_B^{*}(\tau_{i+1})=N_B^{*}(\tau_i)+1$ and $\hat{N}_B^{*}(\tau_{i+1})=\hat{N}_B^{*}(\tau_i)+1$. Further set $R(\tau_{i+1})=V_{i+1}/f(\TP^{*}(\tau_{i+1}))$, $\hat{R}(\tau_{i+1})=V_{i+1}/\hat{f}(\hat{\TP}^{*}(\tau_{i+1}))$ and $J(\tau_{i+1})=J(\tau_{i})-\tau_{i+1}+\tau_i$.

Case 4: $R(\tau_i)=\min\{R(\tau_i),\hat{R}(\tau_i),J(\tau_i)\}$, $R(\tau_i)<\hat{R}(\tau_i)$ and $\sigma^{*}(\tau_i)=0$. Set $\TP^{*}(\tau_{i+1})=\TP^{*}(\tau_{i})$, $\hat{\TP}^{*}(\tau_{i+1})=\hat{\TP}^{*}(\tau_{i})$, $\sigma^{*}(\tau_{i+1})=1$, $\hat{\sigma}^{*}(\tau_{i+1})=\hat{\sigma}^{*}(\tau_{i+1})$, $N_B^{*}(\tau_{i+1})=N_B^{*}(\tau_i)$ and $\hat{N}_B^{*}(\tau_{i+1})=\hat{N}_B^{*}(\tau_i)$. Further set $R(\tau_{i+1})=B_{N_{B}^{*}(\tau_{i+1})+1}$, $\hat{R}(\tau_{i+1})=\hat{R}(\tau_i)-\tau_{i+1}+\tau_i$ and $J(\tau_{i+1})=J(\tau_{i})-\tau_{i+1}+\tau_i$.

Case 5: $R(\tau_i)=\min\{R(\tau_i),\hat{R}(\tau_i),J(\tau_i)\}$, $R(\tau_i)<\hat{R}(\tau_i)$ and $\sigma^{*}(\tau_i)=1$. Set $\TP^{*}(\tau_{i+1})=\TP^{*}(\tau_{i})-1$, $\hat{\TP}^{*}(\tau_{i+1})=\hat{\TP}^{*}(\tau_{i})$, $\sigma^{*}(\tau_{i+1})=0$, $\hat{\sigma}^{*}(\tau_{i+1})=\hat{\sigma}^{*}(\tau_{i})$, $N_B^{*}(\tau_{i+1})=N_B^{*}(\tau_i)+1$ and $\hat{N}_B^{*}(\tau_{i+1})=\hat{N}_B^{*}(\tau_i)$. Further set $R(\tau_{i+1})=V_{i+1}/f(\TP^{*}(\tau_{i+1}))$, $\hat{R}(\tau_{i+1})=(\hat{R}(\tau_i)-\tau_{i+1}+\tau_i)\hat{\sigma}^{*}(\tau_{i+1})+(1-\hat{\sigma}^{*}(\tau_{i+1}))V_{i+1}/\hat{f}(\hat{\TP}^{*}(\tau_{i+1}))$ and $J(\tau_{i+1})=J(\tau_{i})-\tau_{i+1}+\tau_i$.

Case 6: $\hat{R}(\tau_i)=\min\{R(\tau_i),\hat{R}(\tau_i),J(\tau_i)\}$, $R(\tau_i)>\hat{R}(\tau_i)$ and $\hat{\sigma}^{*}(\tau_i)=0$. Set $\TP^{*}(\tau_{i+1})=\TP^{*}(\tau_{i})$, $\hat{\TP}^{*}(\tau_{i+1})=\hat{\TP}^{*}(\tau_{i})$, $\sigma^{*}(\tau_{i+1})=\sigma^{*}(\tau_{i})$, $\hat{\sigma}^{*}(\tau_{i+1})=1$, $N_B^{*}(\tau_{i+1})=N_B^{*}(\tau_i)$ and $\hat{N}_B^{*}(\tau_{i+1})=\hat{N}_B^{*}(\tau_i)$. Further set $R(\tau_{i+1})=R(\tau_i)-\tau_{i+1}+\tau_i$, $\hat{R}(\tau_{i+1})=B_{\hat{N}_{B}^{*}(\tau_{i+1})+1}$ and $J(\tau_{i+1})=J(\tau_{i})-\tau_{i+1}+\tau_i$.

Case 7: $\hat{R}(\tau_i)=\min\{R(\tau_i),\hat{R}(\tau_i),J(\tau_i)\}$, $R(\tau_i)>\hat{R}(\tau_i)$ and $\hat{\sigma}^{*}(\tau_i)=1$. Set $\TP^{*}(\tau_{i+1})=\TP^{*}(\tau_{i})$, $\hat{\TP}^{*}(\tau_{i+1})=\hat{\TP}^{*}(\tau_{i})-1$, $\sigma^{*}(\tau_{i+1})=\sigma^{*}(\tau_{i})$, $\hat{\sigma}^{*}(\tau_{i+1})=0$, $N_B^{*}(\tau_{i+1})=N_B^{*}(\tau_i)$ and $\hat{N}_B^{*}(\tau_{i+1})=\hat{N}_B^{*}(\tau_i)+1$. Further set $R(\tau_{i+1})=(R(\tau_i)-\tau_{i+1}+\tau_i)\sigma^{*}(\tau_{i+1})+(1-\sigma^{*}(\tau_{i+1}))V_{i+1}/f(\TP^{*}(\tau_{i+1}))$, $\hat{R}(\tau_{i+1})=V_{i+1}/\hat{f}(\hat{\TP}^{*}(\tau_{i+1}))$ and $J(\tau_{i+1})=J(\tau_{i})-\tau_{i+1}+\tau_i$.

Note that the remaining cases occur with probability zero as the random variables in the vectors $A$ and $V$ are exponentially distributed.

From the sample path construction we can deduce that
\begin{equation}
\label{difTP2}
\hat{\TP}^{*}(\tau_i)-{\TP^{*}}(\tau_i)=N_B^{*}(\tau_i)-\hat{N}_B^{*}(\tau_i).
\end{equation}
Thus $\hat{\TP}^{*}(\tau_i)={\TP^{*}}(\tau_i)$ if and only if $N_B^{*}(\tau_i)=\hat{N}_B^{*}(\tau_i)$.

We now need to prove that $\hat{\TP}^{*}(\tau_{i+1})>\TP^{*}(\tau_{i+1})$, or $\hat{\TP}^{*}(\tau_{i+1})=\TP^{*}(\tau_{i+1})$ and $0=\hat{\sigma}^{*}(\tau_{i+1})<\sigma^{*}(\tau_{i+1})=1$, or $\hat{\TP}^{*}(\tau_{i+1})=\TP^{*}(\tau_{i+1})$, $\hat{\sigma}^{*}(\tau_{i+1})=\sigma^{*}(\tau_{i+1})$ and $\hat{R}(\tau_{i+1})\geq R(\tau_{i+1})$ in all seven cases.

For case 1 it follows immediately that $\hat{\TP}^{*}(\tau_{i+1})>\TP^{*}(\tau_{i+1})$ if $\hat{\TP}^{*}(\tau_{i})>\TP^{*}(\tau_{i})$. We also find immediately that $\hat{\TP}^{*}(\tau_{i+1})=\TP^{*}(\tau_{i+1})$ and $0=\hat{\sigma}^{*}(\tau_{i+1})<\sigma^{*}(\tau_{i+1})=1$ if $\hat{\TP}^{*}(\tau_{i})=\TP^{*}(\tau_{i})$ and $0=\hat{\sigma}^{*}(\tau_{i})<\sigma^{*}(\tau_{i})=1$. Further, we see that $\hat{\TP}^{*}(\tau_{i+1})= \TP^{*}(\tau_{i+1})$, $\hat{\sigma}^{*}(\tau_{i+1})=\sigma^{*}(\tau_{i+1})=1$ and $\hat{R}(\tau_{i+1})\geq R(\tau_{i+1})$ if $\hat{\TP}^{*}(\tau_{i})= \TP^{*}(\tau_{i})$, $\hat{\sigma}^{*}(\tau_{i})=\sigma^{*}(\tau_{i})=1$ and $\hat{R}(\tau_{i})\geq R(\tau_{i})$. Finally, if $\hat{\TP}^{*}(\tau_{i})= \TP^{*}(\tau_{i})$ and $\hat{\sigma}^{*}(\tau_{i})=\sigma^{*}(\tau_{i})=0$ we get $\hat{\TP}^{*}(\tau_{i+1})= \TP^{*}(\tau_{i+1})$, $\hat{\sigma}^{*}(\tau_{i+1})=\sigma^{*}(\tau_{i+1})=0$ and $\hat{R}(\tau_{i+1})\geq R(\tau_{i+1})$ as $\hat{f}(\hat{\TP}(\tau_{i+1})) \leq f(\TP(\tau_{i+1}))$.

For case 2 recall that $N_B^{*}(\tau_i)=\hat{N}_B^{*}(\tau_i)$ if $\hat{\TP}^{*}(\tau_i)={\TP^{*}}(\tau_i)$, as follows from equation~\eqref{difTP2}. Thus $\hat{R}(\tau_{i+1})= R(\tau_{i+1})$ if $\hat{\TP}^{*}(\tau_{i+1})= \TP^{*}(\tau_{i+1})$. The statement now follows.

For case 3 note that $\hat{f}(\hat{\TP}(\tau_{i+1})) \leq f(\TP(\tau_{i+1}))$ if $\hat{\TP}^{*}(\tau_{i+1})= \TP^{*}(\tau_{i+1})$, which gives $\hat{R}(\tau_{i+1})\geq R(\tau_{i+1})$ in that situation. For all other situations the statement is easily seen to hold.

For case 4 it follows immediately that $\hat{\TP}^{*}(\tau_{i+1})-\hat{\sigma}^{*}(\tau_{i+1})>\TP^{*}(\tau_{i+1})-\sigma^{*}(\tau_{i+1})$. Therefore, $\hat{\TP}^{*}(\tau_{i+1})>\TP^{*}(\tau_{i+1})$, or $\hat{\TP}^{*}(\tau_{i+1})=\TP^{*}(\tau_{i+1})$ and $0=\hat{\sigma}^{*}(\tau_{i+1})<\sigma^{*}(\tau_{i+1})=1$.

For case 5 we get $\hat{\TP}^{*}(\tau_{i+1})>\TP^{*}(\tau_{i+1})$, as follows immediately from the induction hypothesis.

For case 6 note that either $\hat{\TP}^{*}(\tau_{i})>\TP^{*}(\tau_{i})$, or $\hat{\TP}^{*}(\tau_{i})=\TP^{*}(\tau_{i})$ and $0=\hat{\sigma}^{*}(\tau_{i})<\sigma^{*}(\tau_{i})=1$. If $\hat{\TP}^{*}(\tau_{i})>\TP^{*}(\tau_{i})$ we get $\hat{\TP}^{*}(\tau_{i+1})>\TP^{*}(\tau_{i+1})$. Further, if $\hat{\TP}^{*}(\tau_{i})=\TP^{*}(\tau_{i})$ and $0=\hat{\sigma}^{*}(\tau_{i})<\sigma^{*}(\tau_{i})=1$ we get $\hat{\TP}^{*}(\tau_{i+1})=\TP^{*}(\tau_{i+1})$, $\hat{\sigma}^{*}(\tau_{i})<\sigma^{*}(\tau_{i})=1$ and, as $N_B^{*}(\tau_i)=\hat{N}_B^{*}(\tau_i)$ by equation~\eqref{difTP2}, $R(\tau_{i+1})\leq B_{\hat{N}_{B}^{*}(\tau_{i+1})+1} - \tau_{i+1} + \tau_i \leq B_{\hat{N}_{B}^{*}(\tau_{i+1})+1} = \hat{R}(\tau_{i+1})$.

For case 7 note that $\hat{\TP}^{*}(\tau_{i})>\TP^{*}(\tau_{i})$. If $\sigma^{*}(\tau_{i+1})=1$ the statement immediately follows. If $\sigma^{*}(\tau_{i+1})=0$ and $\hat{\TP}^{*}(\tau_{i+1})=\TP^{*}(\tau_{i+1})$ we find $\hat{R}(\tau_{i+1})\geq R(\tau_{i+1})$ as $\hat{f}(\hat{\TP}(\tau_{i+1})) \leq f(\TP(\tau_{i+1}))$.

Finally, it can be verified that the marginal distributions of $\{\TP^{*}(t),\sigma^{*}(t)\}_{t\geq 0}$ and $\{\hat{\TP}^{*}(t),\hat{\sigma}^{*}(t)\}_{t\geq 0}$ are the same as the distribution of $\{\TP(t),\sigma(t)\}_{t\geq 0}$ and $\{\hat{\TP}(t),\hat{\sigma}(t)\}_{t\geq 0}$. For this note that the exponential distribution is memoryless and that $k W \sim {\rm Exp}(\beta/k)$ if $W \sim {\rm Exp}(\beta)$.
\end{proof}

\begin{lemma}
If $\alpha y + (1 - \alpha) z = \alpha' y' + (1 - \alpha') z'$, with
$0 \leq \alpha, \alpha' \leq 1$ and $y' \leq y \leq z \leq z'$, then

${\rm (i)}$ If $g(\cdot)$ is a concave function,
\begin{equation}
\label{lem1gconc}
\alpha g(y) + (1 - \alpha) g(z) \geq \alpha' g(y') + (1 - \alpha') g(z').
\end{equation}

${\rm (ii)}$ If $g(\cdot)$ is a convex function,
\begin{equation}
\label{lem1gconv}
\alpha g(y) + (1 - \alpha) g(z) \leq \alpha' g(y') + (1 - \alpha') g(z').
\end{equation}
\end{lemma}
\begin{proof}
Since $y' \leq y \leq z \leq z'$, there exist
$0 \leq \alpha_y, \alpha_z \leq 1$, such that
$y = \alpha_y y' + (1 - \alpha_y) z'$,
and $z = \alpha_z y' + (1 - \alpha_z) z'$.
It follows from the equality
$\alpha y + (1 - \alpha) z = \alpha' y' + (1 - \alpha') z'$ that
$\alpha' = \alpha \alpha_y + (1 - \alpha) \alpha_z$,
and $1 - \alpha' = \alpha (1 - \alpha_y) + (1 - \alpha) (1 - \alpha_z)$.
Further, if $g(\cdot)$ is concave,
\[
\alpha_y g(y') + (1 - \alpha_y) g(z') \leq g(y),
\]
and
\[
\alpha_z g(y') + (1 - \alpha_z) g(z') \leq g(z).
\]
We may then write
\begin{align*}
\alpha g(y) + (1 - \alpha) g(z)
&\geq \alpha [\alpha_y g(y') + (1 - \alpha_y) g(z')] +
(1 - \alpha) [\alpha_z g(y') + (1 - \alpha_z) g(z')] \\
&= [\alpha \alpha_y + (1 - \alpha) \alpha_z] g(y') +
[\alpha (1 - \alpha_y) + (1 - \alpha) (1 - \alpha_z)] g(z') \\
&= \alpha' g(y') + (1 - \alpha') g(z'),
\end{align*}
which completes the proof for case~${\rm (i)}$.
The inequality in~\eqref{lem1gconv} follows by symmetry.
\end{proof}

\begin{corollary}
\label{kappainc}
For all~$x$, if $a' \leq a < 1$, $b' \geq b > 1$, then

${\rm (i)}$ If $g(\cdot)$ is a concave function,
$\gamma_{a', b'}(x) \leq \gamma_{a, b}(x) \leq 1$ and thus
$\kappa_{a', b'} \geq \kappa_{a, b} \geq 0$.

${\rm (ii)}$ If $g(\cdot)$ is a convex function,
$\gamma_{a', b'}(x) \geq \gamma_{a, b}(x) \geq 1$ and thus
$\chi_{a', b'} \leq \chi_{a, b} \leq 0$.
\end{corollary}
\begin{proof}
Taking $y = a x$, $y' = a' x$, $z = b x$, $z' = b' x$,
$\alpha = (b - 1) / (b - a)$, and $\alpha' = (b' - 1) / (b' - a')$
in~\eqref{lem1gconc}, we obtain for $g(\cdot)$ concave,
\begin{align*}
\frac{(b - 1) g(a x) + (1 - a) g(b x)}{b - a}
&= \alpha g(y) + (1 - \alpha) g(z)\\
&\geq \alpha' g(y') + (1 - \alpha') g(z')
= \frac{(b' - 1) g(a' x) + (1 - a') g(b' x)}{b' - a'},
\end{align*}
which yields the statement for concave $g(\cdot)$.

The assertion for convex $g(\cdot)$ follows by symmetry.
\end{proof}
Let $W$ henceforth be a nonnegative integer-valued random variable
with probability distribution $p(x) = \pr{W = x}$.
For any $y \geq 0$, define
$F(y) = \pr{W \leq y} = \pr{W \leq \lfloor y \rfloor}$, with pseudo inverse
\[
F^{- 1}(u) = \inf\{y: F(y) \geq u\}
\]
for any $u \in [0, 1]$, so that we may write
\[
\expect{g(W)} = \int_{u = 0}^{1} g(F^{- 1}(u)) du,
\]
and in particular
\[
\expect{W} = \int_{u = 0}^{1} F^{- 1}(u) du.
\]
For compactness, denote $\hat{F}^{- 1}(u) = F^{- 1}(u) / \expect{W}$,
\[
x_1(\epsilon_1) =
\frac{1}{\epsilon_1} \int_{u = 0}^{\epsilon_1} \hat{F}^{- 1}(u) du,
\]
and
\[
x_2(\epsilon_2) =
\frac{1}{\epsilon_2} \int_{u = 1 - \epsilon_2}^{1} \hat{F}^{- 1}(u) du.
\]
\begin{lemma}
\label{lem3}
Let $0 < \epsilon_1 \leq F(\expect{W})$,
$0 < \epsilon_2 \leq 1 - F(\expect{W})$, so that
$x_1(\epsilon_1) \leq \hat{F}^{- 1}(\epsilon_1) \leq 1$
and $x_2(\epsilon_2) \geq \hat{F}^{- 1}(1 - \epsilon_2) \geq 1$, with
\[
\epsilon_1 x_1(\epsilon_1) + \epsilon_2 x_2(\epsilon_2) =
\epsilon_1 + \epsilon_2,
\]
or equivalently,
\[
\int_{u = \epsilon_1}^{1 - \epsilon_2} \hat{F}^{- 1}(u) du =
1 - \epsilon_1 - \epsilon_2.
\]

${\rm (i)}$ If $g(\cdot)$ is a concave function,
\begin{equation}
\label{lem3gconc}
(\epsilon_1 + \epsilon_2) \kappa_{x_1(\epsilon_1), x_2(\epsilon_2)} \leq
1 - \frac{\expect{g(W)}}{g(\expect{W})}.
\end{equation}

${\rm (ii)}$ If $g(\cdot)$ is a convex function,
\begin{equation}
\label{lem3gconv}
(\epsilon_1 + \epsilon_2) \chi_{x_1(\epsilon_1), x_2(\epsilon_2)} \geq
1 - \frac{\expect{g(W)}}{g(\expect{W})}.
\end{equation}
\end{lemma}
\begin{proof}
Write
\begin{equation}
\label{expgX}
\expect{g(W)} =
\int_{u = 0}^{\epsilon_1} g(F^{- 1}(u)) du +
\int_{u = \epsilon_1}^{1 - \epsilon_2} g(F^{- 1}(u)) du +
\int_{u = 1 - \epsilon_2}^{1} g(F^{- 1}(u)) du.
\end{equation}

Because of Jensen's inequality we find for concave $g(\cdot)$
\[
\int_{u = \epsilon_1}^{1 - \epsilon_2} g(F^{- 1}(u)) du \leq
(1 - \epsilon_1 - \epsilon_2) g\left(\frac{1}{1 - \epsilon_1 - \epsilon_2}
\int_{u = \epsilon_1}^{1 - \epsilon_2} F^{- 1}(u) du\right) =
(1 - \epsilon_1 - \epsilon_2) g(\expect{W}).
\]

Invoking Jensen's inequality once again,
\begin{eqnarray*}
& &
\int_{u = 0}^{\epsilon_1} g(F^{- 1}(u)) du +
\int_{u = 1 - \epsilon_2}^{1} g(F^{- 1}(u)) du \\
&\leq&
\epsilon_1 g\left(\frac{1}{\epsilon_1}
\int_{u = 0}^{\epsilon_1} F^{- 1}(u) du\right) +
\epsilon_2 g\left(\frac{1}{\epsilon_2}
\int_{u = 1 - \epsilon_2}^{1} F^{- 1}(u) du\right) \\
&=&
\epsilon_1 g(x_1(\epsilon_1) \expect{W}) +
\epsilon_2 g(x_2(\epsilon_2) \expect{W}) \\
&=&
\gamma_{x_1(\epsilon_1), x_2(\epsilon_2)}(\expect{W})
(\epsilon_1 + \epsilon_2) g(\expect{W}) \\
&\leq&
(1 - \kappa_{x_1(\epsilon_1), x_2(\epsilon_2)}) (\epsilon_1 + \epsilon_2)
g(\expect{W}).
\end{eqnarray*}

Substituting the above two inequalities in~\eqref{expgX} we obtain the statement of
the lemma for concave $g(\cdot)$. The assertion for convex $g(\cdot)$ follows from symmetry.
\end{proof}

\begin{lemma}
\label{lem4}
Let $0 < \epsilon_1 \leq F(\expect{W})$,
$0 < \epsilon_2 \leq 1 - F(\expect{W})$, so that
$x_1(\epsilon_1) \leq \hat{F}^{- 1}(\epsilon_1) \leq 1$
and $x_2(\epsilon_2) \geq \hat{F}^{- 1}(1 - \epsilon_2) \geq 1$, with
\[
\epsilon_1 x_1(\epsilon_1) + \epsilon_2 x_2(\epsilon_2) =
\epsilon_1 + \epsilon_2,
\]
or equivalently,
\[
\int_{u = \epsilon_1}^{1 - \epsilon_2} \hat{F}^{- 1}(u) du =
1 - \epsilon_1 - \epsilon_2.
\]

${\rm (i)}$ If $g(\cdot)$ is a concave function,
\[
\kappa_{x_1(\epsilon_1), x_2(\epsilon_2)} \geq
\max\{\kappa_{\hat{F}^{- 1}(\epsilon_1),
1 + \frac{\epsilon_1}{\epsilon_2} (1 - \hat{F}^{- 1}(\epsilon_1))},
\kappa_{1 - \frac{\epsilon_2}{\epsilon_1} (\hat{F}^{- 1}(1 - \epsilon_2) - 1),
\hat{F}^{- 1}(1 - \epsilon_2)}\}.
\]

${\rm (ii)}$ If $g(\cdot)$ is a convex function,
\[
\chi_{x_1(\epsilon_1), x_2(\epsilon_2)} \leq
\min\{\chi_{\hat{F}^{- 1}(\epsilon_1),
1 + \frac{\epsilon_1}{\epsilon_2} (1 - \hat{F}^{- 1}(\epsilon_1))},
\chi_{1 - \frac{\epsilon_2}{\epsilon_1} (\hat{F}^{- 1}(1 - \epsilon_2) - 1),
\hat{F}^{- 1}(1 - \epsilon_2)}\}.
\]
\end{lemma}
\begin{proof}
Observing that
\[
x_2(\epsilon_2) \geq \hat{F}^{- 1}(1 - \epsilon_2),
\]
we obtain
\[
\epsilon_1 x_1(\epsilon_1) \leq
\epsilon_1 + \epsilon_2 - \epsilon_2 \hat{F}^{- 1}(1 - \epsilon_2) =
\epsilon_1 + \epsilon_2 (1 - \hat{F}^{- 1}(1 - \epsilon_2)).
\]

In addition,
\[
x_1(\epsilon_1) \leq \hat{F}^{- 1}(\epsilon_1),
\]
yielding
\[
x_1(\epsilon_1) \leq
\min\{\hat{F}^{- 1}(\epsilon_1), 1 - \frac{\epsilon_2}{\epsilon_1}
(\hat{F}^{- 1}(1 - \epsilon_2) - 1)\}.
\]

Likewise,
\[
x_2(\epsilon_2) \geq
\max\{\hat{F}^{- 1}(1 - \epsilon_2), 1 + \frac{\epsilon_1}{\epsilon_2}
(1 - \hat{F}^{- 1}(\epsilon_1))\}.
\]

Combining the above two inequalities and using Corollary~\ref{kappainc} completes the proof.
\end{proof}

\begin{repproposition}{prop1}
Assume $g(\cdot)$ is concave and $\kappa_{a, b} >0$ for any $a < 1$ and $b > 1$, or $g(\cdot)$ is convex and $\chi_{a, b} < 0$ for any $a < 1$ and $b > 1$. If
\[
\lim_{\rho \uparrow 1} \frac{\expect{g(W)}}{g(\expect{W})} = 1,
\]
then
\[
\frac{W}{\expect{W}} \cond 1 \mbox{ as } \rho \uparrow 1.
\]
\end{repproposition}
\begin{proof}
Take $\delta > 0$ and $\epsilon_1 = F((1 - \delta) \expect{W})$.
Then either $\epsilon_1 = 0$, or $0 < \epsilon_1 \leq F(\expect{W})$
and $x_1(\epsilon) \leq \hat{F}^{- 1}(\epsilon_1) \leq 1 - \delta$.
In the latter case, define $\epsilon_2^* = 1-\hat{F}^{-1}(\expect{W})$,
and observe that
\[
\int_{u = \epsilon_1}^{1 - \epsilon_2^*} \hat{F}^{- 1}(u) du \leq
1 - \epsilon_1 - \epsilon_2^*,
\]
while
\[
\int_{u = \epsilon_1}^{1} \hat{F}^{- 1}(u) du \geq 1 - \epsilon_1.
\]
Hence, by continuity, there must exist
an $\epsilon_2 \in (0, \epsilon_2^*)$ with $x_2(\epsilon_2) > 1$ and
\[
\int_{u = \epsilon_1}^{1 - \epsilon_2} \hat{F}^{- 1}(u) du =
1 - \epsilon_1 - \epsilon_2,
\]
so that the assumptions of Lemmas~\ref{lem3} and~\ref{lem4} are
satisfied.
Applying these two lemmas then yields for concave $g(\cdot)$
\[
\kappa_{\hat{F}^{- 1}(\epsilon_1),
1 + \epsilon_1 (1 - \hat{F}^{- 1}(\epsilon_1))} \leq
\kappa_{\hat{F}^{- 1}(\epsilon_1),
1 + \frac{\epsilon_1}{\epsilon_2} (1 - \hat{F}^{- 1}(\epsilon_1))} \to 0
\mbox{ as } \rho \uparrow 1.
\]

This means that $\epsilon_1 = \pr{W \leq (1 - \delta) \expect{W}} \to 0$
as $\rho \uparrow 1$. A similar argument shows that
$\pr{W \geq (1 + \delta) \expect{W}} \to 0$ as $\rho \uparrow 1$. It now follows from the definition of convergence in probability that $\frac{W}{\expect{W}}$ converges to $1$ in probability. Hence we conclude that $\frac{W}{\expect{W}} \cond 1$ as $\rho \uparrow 1$ if $g(\cdot)$ is concave.

The proof for convex~$g(\cdot)$ follows by symmetry.
\end{proof}


%
%

\end{document}